\documentclass[reqno]{amsart}

\usepackage{amsmath,amssymb,amscd,accents,enumitem}
\usepackage{graphicx}
\usepackage{microtype}
\usepackage[utf8]{inputenc} 

\usepackage[T2A,T1]{fontenc}

\usepackage{etoolbox}
\pretocmd{\section}{\numberwithin{Cor}{section}}{}{}
\pretocmd{\subsection}{\numberwithin{Cor}{subsection}}{}{}
\pretocmd{\subsubsection}{\numberwithin{Cor}{subsubsection}}{}{}

\DeclareSymbolFont{cyrillic}{T2A}{cmr}{m}{n}
\SetSymbolFont{cyrillic}{bold}{T2A}{cmr}{bx}{n}
\DeclareMathSymbol{\B}{\mathord}{cyrillic}{193}

\DeclareGraphicsExtensions{.eps} \usepackage{mathrsfs}
\usepackage[mathcal]{eucal} 
\usepackage{esint}
\usepackage[numbers,sort]{natbib}

\usepackage[pdfencoding=unicode, psdextra, pdfusetitle]{hyperref}
\hypersetup{
  colorlinks   = true,
  linkcolor    = [RGB]{240,50,240},
  citecolor    = [RGB]{100,0,200}
}
\usepackage{bookmark}
\renewcommand{\paragraph}[1]{\medskip\noindent#1}
\renewcommand{\subparagraph}[1]{\medskip\noindent\textit{#1}}
\newcommand{\mycite}[2]{\cite[#1]{#2}}

\newcommand{\iref}[2][]{%
  \def\tempA##1 ##2\relax{%
    \ifstrequal{##1}{Condition}%
        {Condition~(\ref{#1#2})}{%
    \ifstrequal{##1}{Assumption}%
        {Assumption~(\ref{#1#2})}{%
    \ifstrequal{##1}{Property}%
        {Property~(\ref{#1#2})}{%
    \ifstrequal{##1}{Case}%
        {Case~(\ref{#1#2})}{%
    \ifstrequal{##1}{Equation}%
        {\eqref{#1#2}}{%
    \ifstrequal{##1}{Section}
        {\S\ref{#1#2}}{%
    \ifstrequal{##1}{Subsection}
        {\S\ref{#1#2}}{%
    \ifstrequal{##1}{Subsubsection}
        {\S\ref{#1#2}}%
        {##1~\ref{#1#2}}%
  }}}}}}}}%
  \expandafter\tempA\string#2\relax
}

\theoremstyle{plain}
\newtheorem{Thm}{Theorem}
\newtheorem{CorThm}[Thm]{Corollary}
\newtheorem{PropThm}[Thm]{Proposition}
\newtheorem{Cor}{Corollary}
\newtheorem{ThmCor}[Cor]{Theorem}
\newtheorem{Prop}[Cor]{Proposition}
\newtheorem{Construction}[Cor]{Construction}
\newtheorem{Lem}[Cor]{Lemma}

\theoremstyle{definition}
\newtheorem{Def}[Cor]{Definition}

\newtheorem{Notation}[Cor]{Notation}

\theoremstyle{remark}
\newtheorem{Rem}[Cor]{Remark}

\theoremstyle{plain}

\begin{document}
\title{Sharp zero estimates for trajectories of polynomial vector fields}
\author{Gal Binyamini, Yuval Salant} 

\address{Weizmann Institute of Science, Rehovot, Israel}
\email{gal.binyamini@weizmann.ac.il}
 \thanks{
  Funded by the European Union (ERC, SharpOS, 101087910), and by the
  ISRAEL SCIENCE FOUNDATION (grant No. 2067/23).}

\date{\today}

\begin{abstract}
    If $f$ is a tuple of functions satisfying an algebraic ODE and $P\in{\mathbb C}(x)[f]$, it is common in applications to transcendental number theory to consider upper bounds for the order of zero of $P(x,f)$ at a given point in terms of $\deg_x P,\deg_f P$. Nesterenko introduced a condition known as the \emph{D-property}, which holds in many applications, and proved essentially optimal bounds under this condition.

    We prove a global form of this result, where the field ${\mathbb C}(x)$ is replaced by a number field $K$ and $\deg_x P$ is replaced by the logarithmic height $\operatorname{h}(P)$. Under an analogous \emph{global D-property}, we prove essentially optimal bounds for the \emph{number of zeroes}, counted with multiplicities, in a fixed compact set. Applications of this result to point-counting theorems are developed in a separate joint paper with Hirata-Kohno and Kawashima.
\end{abstract}

\maketitle

\section{Introduction}

Let $p\in\mathbb C[x_1,...,x_n]$ and let $\xi$ be a polynomial vector field in ${\mathbb C}^n$ defined over $\overline{\mathbb Q}$. Let $D_1\subset\mathbb C$ denote the unit disc and let $\gamma:\overline{D_1}\to\mathbb C^n$ be a trajectory of $\xi$. We assume throughout that $\gamma$ is not contained in any proper algebraic variety.
We recall Nesterenko's D-property and introduce a global analogue.

\begin{Def}[D-property and global D-property, cf.~\protect{\cite[p.1346]{Nesterenko1996ModularFunctionsAndTranscendenceQuestions}}]\label{Definition D property and global D-property}
We say that $\xi$ satisfies the D-property at a point $\omega\in\mathbb C^n$ if all $\xi$-invariant prime ideals in the local ring at $\omega$ have a common nonzero element.
Similarly, we say that $\xi$ satisfies the \emph{global} D-property if all $\xi$-invariant prime ideals have a common nonzero element.
\end{Def}

\begin{Rem}
Our main theorems hold under a weaker assumption on $\gamma$, the \emph{arithmetic D-property}; see \iref{Subsection the arithmetic D-property}.
\end{Rem}

Denote
\[ \operatorname{ord}_{\gamma,z_0} p:=\operatorname{ord}_{z_0}p\circ\gamma.
\]

Nesterenko proves the following theorem:

\begin{ThmCor}[\mycite{Chapter 3, Theorem 2.3}{NesterenkoPhilippon2001IntroductionToAlgebraicIndependenceTheory}]\label{Theorem introduction; bound on nesterenko's ord}
For any $z_0\in D_1$, if $\xi$ satisfies the D-property at $z_0$, then there exists a constant $C$ depending only on $\xi,\gamma(z_0)$ such that for any $p\in\mathbb C[x_1,...,x_n]$ we have
\[
    \operatorname{ord}_{\gamma,z_0} p
        \leq
    C\cdot(\deg p)^n
.\]
\end{ThmCor}

We prove an analogue of \iref{Theorem introduction; bound on nesterenko's ord}, replacing the local order at a point with a global notion of order, as follows.

\begin{Def}[order]
Let $f:\mathbb C^n\to\mathbb C$ be a function that is not identically zero. The order of $f$ with respect to $\gamma$ is defined as
\[
    \operatorname{ord}_{\gamma}f:=-\ln\max_{z\in\overline{D_1}}|f(\gamma(z))|
.\]
\end{Def}

Our main theorem is the following. For $p\in\mathbb Z[x_1,...,x_n]-\{0\}$, denote by $\operatorname{h}(p)$ the logarithm of the largest absolute value of a coefficient of $p$, and set $\operatorname{t}(p):=\deg p+\operatorname{h}(p)$.

\begin{Thm}[see \iref{Corollary the lower bound main theorem; main remark}]\label{Theorem introduction; main theorem}
Suppose $\xi$ satisfies the global D-property. Then there exists a constant $C$ depending only on $\xi,\gamma$ such that for any $p\in\mathbb Z[x_1,...,x_n]-\{0\}$, we have 
\[
    \operatorname{ord}_\gamma p \leq
    C\cdot\operatorname{t}(p)^n.
\]
\end{Thm}

We note two corollaries. First, we show that the order bound implies a bound on the number of zeroes, thus generalizing Nesterenko's local result in this context.

\begin{CorThm}[see \iref{Corollary number of zeroes with multiplicity}]
  Suppose $\xi$ satisfies the global D-property. Then there is a constant $C$ depending only on $\xi,\gamma$ such that for any $p\in\mathbb Z[x_1,...,x_n]$ we have
  \begin{equation*}
    \# \text{zeroes of }p\circ\gamma\text{ in }\overline{D_1}\text{ counted with multiplicities}  \le C\cdot\operatorname{t}(p)^n.
  \end{equation*}
\end{CorThm}

While generalizing Nesterenko's result, this bound on the number of zeroes does not suffice for our intended applications to transcendental number theory, developed in our joint work with Hirata-Kohno and Kawashima in~\cite{V1BinyaminiHirata_KohnoKawashimaSalant2026CountingTheoremsForAlgebraicRelations}. For these applications, we require a more quantitative statement, showing that a polynomial is not only nonzero but also bounded from below outside a small union of discs. A result of this type follows from our order bound, as stated below.

\begin{CorThm}[see \iref{Corollary the lower bound main theorem; main corollary}]\label{Corollary introduction; main corollary}
Suppose $\xi$ satisfies the global D-property. Then there exists a constant $C$ depending only on $\xi,\gamma$ such that the following holds.
For any $p\in\mathbb Z[x_1,...,x_n]-\{0\}$ and any $0<\varepsilon<\frac1e$, there exists a collection of discs $\{D^i\}_{i=1}^a$ such that
\begin{enumerate}
    \item The sum of their diameters is less than $\varepsilon$.
    \item The number of discs is at most $C\cdot\operatorname{t}(p)^n$.
    \item For every $\omega\in\gamma(\overline{D_\frac1e}-\bigcup_i D^i)$ we have
    \[
    \ln|p(\omega)|
        \geq
    -C\cdot\ln\frac1\varepsilon\cdot\operatorname{t}(p)^n
    .\]
\end{enumerate}
\end{CorThm}

Our main theorem and its corollaries in \iref{Section main theorem} are formulated for polynomials over any number field $K$ rather than $\mathbb Z$. To simplify the introduction, we have restricted it to the case of $\mathbb Z$-coefficients, but no substantial issues arise in passing to this more general case.

\subsection*{AI Disclosure}

ChatGPT was used to preform basic English and typographic edits and minor formatting. No substantive mathematical work was preformed with the help of AI systems.

\section{Preliminaries}

\subsection{Overview of the Bernstein index}\label{Subsection bernstein overview}

The Bernstein index is a useful complex-analytic notion that measures the growth rate of functions on a compact set. For basic results on the Bernstein index and its properties, we refer the reader to \cite{NovikovYakovenko1997AComplexAnalogOfRolle'sTheoremAndPolynomialEnvelopesOfIrreducibleDifferentialEquationsInTheComplexDomain,IlyashenkoYakovenko1996CountingRealZerosOfAnalyticFunctionsSatisfyingLinearOrdinaryDifferentialEquations}.
The Bernstein index generalizes the $\deg$ property of polynomials to larger classes of functions (usually holomorphic functions).

In general, the Bernstein index is defined for a pair of sets $K\subseteq U\subseteq\mathbb C$ where $U$ is a bounded, connected, and simply connected open set with a nice boundary and $K$ is a compact subset. In our setting, we work with $K:=\overline{D_\frac1e},U:=D_1$.
Let $f$ be a continuous function $\overline{D_1}\to\mathbb C$. We define the Bernstein index of $f$ as follows:

\begin{Def}[Bernstein index]
\[
    \B_{\frac1e,1}(f)
        :=
    \ln\frac{\max_{z\in\overline{D_1}}|f(z)|}{\max_{z\in\overline{D_\frac1e}}|f(z)|}
.\]
\end{Def}

As we shall see here and in \iref{Section bernstein index theorems}, the Bernstein index behaves well under products, maxima, zero counting, and differentiation.
Let $f,f_1,...,f_k$ be functions holomorphic in $U$ and continuous on $\overline{U}$. We recall some important results concerning their Bernstein indices:

\begin{Lem}[\mycite{Lemma 1}{IlyashenkoYakovenko1996CountingRealZerosOfAnalyticFunctionsSatisfyingLinearOrdinaryDifferentialEquations}]\label{Lemma bound on number of zeroes of f given bernstein index}
The number of zeroes of $f$ in $K$ (counted with their multiplicities), denoted by $N_K(f)$, satisfies
\[
    N_K(f)\leq C_{K,U}\cdot\B_{K,U}(f)
,\]
where $C_{K,U}$ is some constant depending only on $K,U$.
\end{Lem}

\begin{Lem}[\mycite{Lemma 3}{IlyashenkoYakovenko1996CountingRealZerosOfAnalyticFunctionsSatisfyingLinearOrdinaryDifferentialEquations}]\label{Lemma bernstein index lower bound classic}
For any $0<h<\frac1e$, there is a collection of discs $\{D^i\}_{i=1}^a$ whose diameters sum to less than $h$ such that
\[
    \forall x\in K-\bigcup_i D^i:|f(x)|\geq\frac{\max_{z\in K}|f(z)|}{e^{C_{K,U}\cdot\ln\frac1h\cdot\B_{K,U}(f)}}
,\]
where $C_{K,U}$ is some constant depending only on $K,U$.
\end{Lem}

\begin{Rem}\label{Remark number of discs}
The proof of \cite[Lemma 3]{IlyashenkoYakovenko1996CountingRealZerosOfAnalyticFunctionsSatisfyingLinearOrdinaryDifferentialEquations} implies that the number of discs, $a$, is at most $\gamma(U,\overline V)\B_{K,U}(f)$, where $\gamma(U,\overline V)$ is a constant depending only on $U$ and on $\overline V$. Here $V$ can be chosen at the beginning of the proof depending only on $K,U$ and not on $f$ (so we can treat $\gamma(U,\overline V)$ as a constant depending only on $K,U$). Thus, we can choose $C_{K,U}$ large enough that
\[
    a\leq C_{K,U}\cdot\B_{K,U}(f)
.\]
\end{Rem}

\begin{PropThm}[see \iref{Corollary bernstein index of product}]
\[
    \B_{K,U}(\prod_{i=1}^k f_i)
        \leq
    C_{K,U}\cdot\sum_{i=1}^k \B_{K,U}(f_i)
\]
where $C_{K,U}$ is some constant depending only on $K,U$ (but not on $k$).
\end{PropThm}

\begin{Rem}
This corollary improves \cite[Lemma 3]{NovikovYakovenko1997AComplexAnalogOfRolle'sTheoremAndPolynomialEnvelopesOfIrreducibleDifferentialEquationsInTheComplexDomain}, which has an additional multiplicative $\ln k$ factor.
\end{Rem}

\begin{PropThm}[see \iref{Proposition derivative-function ratio}]
Let $g\not\equiv 0$ be a holomorphic function defined on $D_{1+\varepsilon}$ for some $\varepsilon>0$, and let $h<1$. Then there exists $A\subseteq\overline{D_\frac1e}$ with $\operatorname{Area}(A)\geq h\cdot\operatorname{Area}\left(\overline{D_\frac1e}\right)$ such that
\[
    \forall z\in A:|g'(z)|\leq C_h\cdot\B_{\frac1e,1}(g)|g(z)|
.\]
where $C_h$ is some constant depending only on $h$.
\end{PropThm}

\subsection{Our setting}

We fix $n\in\mathbb N^+$ and work with objects in $\mathbb C^{n+1}-\{0\}$ (which we think of as the projective space $\mathbb{CP}^n$).

We fix a number field $K$. We think of $K$ as a subset of $\mathbb C$ (i.e., we fix a specific embedding of $K$ into $\mathbb C$). We work mainly with homogeneous ideals contained in $K[x_0,...,x_n]$. We call these ideals over $K$ and their zero loci in $\mathbb C^{n+1}-\{0\}$ varieties defined over $K$.

We fix a homogeneous polynomial vector field $\xi$ on $\mathbb C^{n+1}-\{0\}$ with coefficients in $K$. That is, $\xi=p_0(x_0,...,x_n)\frac{\partial}{\partial x_0}+...+p_n(x_0,...,x_n)\frac{\partial}{\partial x_n}$ for some homogeneous polynomials of equal degree $p_1,...,p_n\in K[x_0,...,x_n]$.

\begin{Notation}
$D_r$ is the open disc $\{z\in\mathbb C:|z|<r\}$ and $\overline{D_r}$ is its closure.
\end{Notation}

We fix a small compact holomorphic trajectory $\gamma$ of $\xi$. We assume:
\begin{enumerate}
    \item $\gamma$ is a holomorphic map $\gamma:\overline{D_1}\rightarrow\mathbb C^{n+1}-\{0\}$ (this means that $\gamma$ is defined and holomorphic on some open neighborhood of $\overline{D_1}$).
    \item $\gamma$ is tangent to $\xi$ at every point. That is, $\forall z\in\overline{D_1}\ \exists a_z\in\mathbb C:\quad a_z\cdot\frac{d\gamma}{dz}(z)=\xi_{\gamma(z)}$. We note that this implies that $a_z$ is holomorphic as a function of $z\in\overline{D_1}$.
    \item \label{Assumption gamma non degenerate derivative}$\forall z\in\overline{D_1}:\ \frac{d\gamma}{dz}(z)\neq 0$.
    \item \label{Assumption gamma lands in small affine ball}We require that $\forall z\in\overline{D_1}:\ \gamma(z)\in\{1\}\times\overline{D_\frac12}^n$ (that is, the first coordinate is $1$ and the rest are at most $\frac12$ in absolute value).
\end{enumerate}

\begin{Rem}
The theorems in this article would hold even without \iref{Assumption gamma non degenerate derivative}
 and \iref{Assumption gamma lands in small affine ball}; we assume them for simplicity.
\end{Rem}

\begin{Notation}
For any $r\in[0,1]$, we define $\Gamma_r:=\gamma(\overline{D_r})$.
We use the notation $\Gamma:=\Gamma_\frac1e$.
\end{Notation}

\begin{Rem}[constants]
Throughout the article, we establish various bounds. Our goal is to find an upper bound for $\operatorname{ord}_{\Gamma}I$ that depends on $\operatorname{t}(I)$ and on the fixed parameters listed below. We seek the best dependence on $\operatorname{t}(I)$ that we can achieve. We do not track the dependence on the fixed parameters. For this reason, whenever a parameter depends only on
\[
     K,\xi,\gamma,\Gamma,n
,\]
we denote it by $\mathrm{const}$. Each occurrence of $\mathrm{const}$ refers to a parameter depending only on the fixed parameters (the value of $\mathrm{const}$ may vary between occurrences, as with $O(1)$, with dependence on the fixed parameters allowed).
Each $\mathrm{const}$ is nonnegative.
\end{Rem}

\begin{Rem}
The theorems in this article would also hold in the following more general context:
For any $r_1<r_2\in\mathbb R_{>0}$, we could have assumed $\gamma$ to be defined on $\overline{D_{r_2}}$ and worked with $\Gamma_{r_1}$ instead of $\Gamma$ and with $\B_{r_1,r_2}$ instead of $\B_{\frac1e,1}$. In this case, all the resulting constants would depend on $r_1,r_2$ as well.
\end{Rem}

\subsection{Notions and notation from intersection theory}

\subsubsection{Basic notation and definitions}

In this subsection, we introduce notation and definitions from elimination theory (see \cite[Chapter 3, \S4, Algebraic fundamentals]{NesterenkoPhilippon2001IntroductionToAlgebraicIndependenceTheory}).

\begin{Rem}
Throughout, we work with a fixed number field $K$. For simplicity, we regard number fields as embedded in $\mathbb C$. Equivalently, we fix a specific embedding of each number field into $\mathbb C$.
\end{Rem}

\begin{Notation}
Let:
\begin{itemize}
    \item $\mathcal O_ K$ be the ring of integers of the number field $K$.
    \item $\operatorname{Gal}(K,\mathbb Q)$ be the Galois group of $K/\mathbb Q$.
    \item $\mathfrak P_{ K}$ be the set of prime ideals of $\mathcal O_ K$.
    \item $k\neq 0$ be an element of $K$.
    \item $|k|$ be the standard absolute value on $\mathbb C$ applied to $k$.
\end{itemize}

We define $\mathcal V$ to be the collection of absolute values on $K$ as in \cite[Chapter 3, \S4, Algebraic fundamentals]{NesterenkoPhilippon2001IntroductionToAlgebraicIndependenceTheory}, including both non-archimedean and archimedean absolute values. We note that:
\begin{itemize}
    \item Any absolute value applied to $0$ returns $0$.
    \item $\prod_{v\in\mathcal V}|k|_v=1$.
    \item \resizebox{.99\linewidth}{!}{$|\cdot|:=|\cdot|_{\operatorname{id}}:=\text{the absolute value that comes from the standard absolute value in }\mathbb C$}.
\end{itemize}
\end{Notation}

\begin{Def}\label{Definition basic definitions}
Let:
\begin{itemize}
\item $\omega:=(\omega_0,...,\omega_n),a:=(a_0,...,a_n),b:=(b_0,...,b_n)\in\mathbb C^{n+1}-\{0\}$.
    \item $p:=\sum_{i_0,...,i_n}a_{i_0,...,i_n}x_0^{i_0}\cdots x_n^{i_n}\in K[x_0,...,x_n]-\{0\}$ be a nonzero homogeneous polynomial.
    \item $v\in\mathcal V$ be an absolute value on $K$.
\end{itemize}

We define:
\begin{enumerate}
\item $|\omega|_v:=\max_i|\omega_i|_v$.
    \item $|p|_v:=\max_{i_0,...,i_n}|a_{i_0,...,i_n}|_v$.
    \item $\operatorname{H}(p):=\prod_v|p|_v$.
    \item $\operatorname{h}(p):=\ln\operatorname{H}(p)$.
    \item $\operatorname{t}(p):=\deg p+\operatorname{h}(p)$.
    \item $\|p\|_\omega:=|p(\omega)|\cdot|p|^{-1}\cdot|\omega|^{-\deg p}$.
    \item $\|a-b\|:=\max_{i\neq j}|a_i b_j-a_j b_i|\cdot|a|^{-1}\cdot|b|^{-1}$.
\end{enumerate}
\end{Def}

\begin{Rem}
\begin{align*}  
        \forall\lambda_1,\lambda_2\in\mathbb C:
        &\\&\|\lambda_1 p\|_{\lambda_2 \omega}=\|p\|_\omega,\ \|\lambda_1 a-\lambda_2 b\|=\|a-b\|,\operatorname{h}(\lambda_1 p)=\operatorname{h}(p)
.\end{align*}
\end{Rem}

\begin{Rem}\label{Remark projective vs affine distance}
    If $a:=(1,a_1,\ldots,a_n), b:=(1,b_1,\ldots,b_n)$ and each $|a_i|$ is $\leq\frac12$ for $i\geq 1$, then $\|a-b\|=\|a-b\|_\infty\cdot|b|^{-1}$.
\end{Rem}

\begin{proof}
Choosing $j:=0$ gives
\[
    \|a-b\|\cdot|b|
        =
    \max_{i\neq j}|a_i b_j-a_j b_i|
        \geq
    \max_{i\neq 0}|a_i\cdot 1 - 1\cdot b_i|
        =
    \|a-b\|_\infty
.\]

Let $i\neq 0,j\neq 0$. Let $\varepsilon_1:=a_i-b_i,\varepsilon_2:=a_j-b_j$. Since $a_i,a_j\leq\frac12$, we have
\begin{align*}
    |a_i b_j - a_j b_i| 
    &= |a_i (a_j - \varepsilon_2) - a_j (a_i - \varepsilon_1)| 
    = |-a_i \varepsilon_2 + a_j \varepsilon_1| \\
    &\leq \frac{|\varepsilon_1| + |\varepsilon_2|}{2} 
    \leq \max(|\varepsilon_1|, |\varepsilon_2|)
    \leq \|a - b\|_\infty
.\end{align*}
\end{proof}

\subsubsection{Definitions related to ideals and the Chow form}

\begin{Notation}[zero locus of an ideal]
Given a homogeneous ideal $I\subseteq R[x_0,...,x_n]$ over any ring $R\subseteq\mathbb C$, we denote its zero locus by
\[
    \operatorname{V}(I):=\{\omega\in\mathbb C^{n+1}-\{0\}:\forall p\in I(p(\omega)=0)\}
.\]
\end{Notation}

\begin{Def}[dimension of a homogeneous ideal]
The dimension $\dim$ of a homogeneous ideal is defined as the dimension of its zero locus, interpreted as a projective variety (i.e., one less than the affine dimension).
\end{Def}

\begin{Def}[unmixed ideal]
 A homogeneous ideal $I\subseteq K[x_0,...,x_n]$ is said to be unmixed if, in its reduced primary decomposition $I=Q_1\cap...\cap Q_k$, all the radicals $\sqrt{Q_i}$ have the same dimension.
\end{Def}

\begin{Def}[Chow form]\label{Definition chow form}
Let:
\begin{enumerate}
\item
\begin{itemize}
    \item $I\subseteq K[x_0,...,x_n]$ be a homogeneous unmixed ideal.
    \item $\{u_{1,0},...,u_{\dim I+1,n}\}$ be a collection of $(\dim I+1)\cdot(n+1)$ variables.
    \item $\{L_i:=u_{i,0}x_0+...+u_{i,n}x_n\}_i$ be a collection of bilinear forms in the $u_{i,j}'s$ and $x_{k}$'s.
\end{itemize}
\end{enumerate}

We define:
\begin{enumerate}[resume]
    \item $\bar I\subseteq K[u_{i,j}]_{i,j}$ is the set of all polynomials $p(u_{1,0},...,u_{\dim I+1,n})\in K[u_{i,j}]_{i,j}$ for which $\exists M_p\in\mathbb N$ such that $\forall i:\ x_i^{M_p}p\in(I,L_1,...,L_{\dim I+1})\subseteq K[u_{i,j}]_{i,j}[x_k]_k$.
\end{enumerate}

By \cite[Lemma 5]{Nesterenko1977estimatesForTheOrdersOfZerosOfFunctionsOfACertainClassAndApplicationsInTheTheoryOfTranscendentalNumbers}, $\bar I$ is principal.
\begin{enumerate}[resume]
\item The Chow form of $I$ is defined as the generator of $\bar I$. We denote it by $\operatorname{Chow}(I)$ and note that it is a homogeneous polynomial in $K[u_{i,j}]_{i,j}$, defined only up to multiplication by a scalar in $K-\{0\}$.
\end{enumerate}
\end{Def}

\begin{Notation}
Let:
\begin{itemize}
    \item \[S^{(i)}:=
    \begin{pmatrix}
    0 & s^{(i)}_{0,1} & \cdots & s^{(i)}_{0,n} \\
    -s^{(i)}_{0,1} & 0 & \ddots & \vdots \\
    \vdots & \ddots & 0 & s^{(i)}_{n-1,n} \\
    -s^{(i)}_{0,n} & \cdots & -s^{(i)}_{n-1,n} & 0
    \end{pmatrix}
    \]
be a collection (over $i$) of skew-symmetric $(n+1)\times(n+1)$ matrices whose entries are variables.
    \item Let $\omega:=
    \begin{pmatrix}
    \omega_0 \\
    \vdots \\
    \omega_n
    \end{pmatrix}
    $ be a column vector of variables.
    \item $p(u_{1,0},...,u_{\dim I+1,n})\in K[u_{i,j}]_{i,j}$ be a polynomial.
\end{itemize}

We define $\chi(p)$ to be the following polynomial in $K[s_{j,k}^{(i)}]_{i,j<k}[\omega_l]_l$:
\[
    \chi(p)
        :=
    p(S^{(1)}\cdot\omega,...,S^{(\dim I+1)}\cdot\omega)
    \in
     K[s_{j,k}^{(i)}]_{i,j<k}[\omega_l]_l
.\]

If, in addition, we are given a specific $\omega_0\in\mathbb C^{n+1}-\{0\}$, we use the following notation for substituting $\omega_0$ into $\chi(p)$:
\[
    \chi_{\omega_0}(p)
        :=
    \chi(p)(\omega_0)
        :=
    p(S^{(1)}\cdot\omega_0,...,S^{(\dim I+1)}\cdot\omega_0)
    \in
     K[s_{j,k}^{(i)}]_{i,j<k}
.\]
\end{Notation}

\begin{Def}\label{definition of deg I, h(I), t(I), I(omega)}
Let $I\subseteq K[x_0,...,x_n]$ be a homogeneous unmixed ideal, and let $\omega_0\in\mathbb C^{n+1}-\{0\}$. We define:
\begin{enumerate}
    \item $\deg I:=\deg_{u_{1,0},...,u_{1,n}}\operatorname{Chow}(I)$.
    \item $\operatorname{h}(I):=\operatorname{h}(\operatorname{Chow}(I))$.
\item $\operatorname{t}(I):=\deg I+\operatorname{h}(I)$.
    \item $|I|(\omega_0):=|I(\omega_0)|:=|\chi_{\omega_0}(\operatorname{Chow}(I))|\cdot|\operatorname{Chow}(I)|^{-1}\cdot|\omega_0|^{-(\dim I + 1)\deg I}$,
where $|\chi_\omega(\operatorname{Chow}(I))|$ means the absolute value as a polynomial in $K[s_{j,k}^{(i)}]_{i,j<k}$.
\end{enumerate}
\end{Def}

\begin{Rem}
\leavevmode
\begin{itemize}
    \item $\operatorname{h}(I)$ and $\deg I$ do not depend on the choice of $\operatorname{Chow}(I)$ and so are well defined.
    \item The normalization in the definition of $I(\omega)$ ensures that it does not depend on the choice of the Chow form or of a representative of $\omega_0$ (up to scalar multiplication). That is, it is well defined and $\forall\lambda\in\mathbb C-\{0\}:\ |I(\lambda\omega_0)|=|I(\omega_0)|$.
    \item When $I$ is radical, $\deg I$ is the degree of $V(I)$ in the usual sense.
\end{itemize}
\end{Rem}

\begin{Rem}
We sometimes work with varieties instead of ideals. We extend these notions to a variety by considering its corresponding radical ideal.
\end{Rem}

\begin{Def}[distance between a point and an ideal]
Let $\omega\in\mathbb C^{n+1}-\{0\}$ and let $I\subseteq K[x_0,...,x_n]$ be a homogeneous ideal. We define the distance between $\omega$ and $\operatorname{V}(I)$ as
\[
    \rho_{\omega,I}
        :=
    \rho_{\omega,\operatorname{V}(I)}
        :=
    \min_{x\in\operatorname{V}(I)}\|\omega-x\|
.\]
\end{Def}

\subsubsection{Important theorems from intersection theory}

In this subsubsection, we quote important results from \cite[Chapter 3, \S4, Algebraic fundamentals]{NesterenkoPhilippon2001IntroductionToAlgebraicIndependenceTheory}, which we will use throughout.

Let $I\subseteq K[x_0,...,x_n]$ be a homogeneous unmixed ideal, let $\mathfrak p\subseteq K[x_0,...,x_n]$ be a homogeneous prime ideal, let $Q\in K[x_0,...,x_n]$ be homogeneous, and let $\omega\in\mathbb C^{n+1}-\{0\}$. Then there exists a constant $C$ depending only on $K,n$ such that the following lemmas hold:

\begin{Lem}[see \mycite{Chapter 3, Proposition 4.7}{NesterenkoPhilippon2001IntroductionToAlgebraicIndependenceTheory}]\label{Lemma Nesterenko primary decomposition of ideal}
 Consider the reduced primary decomposition of $I$, with prime radicals $\mathfrak p_i$ of multiplicity $k_i$. Then
\begin{enumerate}
 \item $\sum k_i\deg\mathfrak p_i=\deg I$.
 \item $\sum k_i\operatorname{h}(\mathfrak p_i)\leq\operatorname{h}(I)+C\cdot\deg I$.
 \item $\sum k_i\ln|\mathfrak p_i(\omega)|\leq\ln|I(\omega)|+n^3\cdot\deg I$.
\end{enumerate}
\end{Lem}

\begin{Lem}[see \mycite{Chapter 3, Proposition 4.8}{NesterenkoPhilippon2001IntroductionToAlgebraicIndependenceTheory}]\label{Lemma Nesterenko principal ideal}
If $I=(Q)$, then:
\begin{enumerate}
 \item $\deg I=\deg Q$.
 \item $\operatorname{h}(I)\leq\operatorname{h}(Q)+C\cdot\deg Q$.
 \item $\ln|I(\omega)|\leq\ln\|Q\|_\omega+C\cdot\deg Q$.
\end{enumerate}
\end{Lem}

\begin{Lem}[see \mycite{Chapter 3, Proposition 4.11}{NesterenkoPhilippon2001IntroductionToAlgebraicIndependenceTheory}]\label{Lemma Nesterenko proposition intersection of ideal and polynomial}
Assume that $\dim\mathfrak p\geq 1$ and that $Q\not\in\mathfrak p$. Then there exists a homogeneous unmixed ideal $J$ of dimension $\dim\mathfrak p -1$, whose zero locus is the intersection of those of $\mathfrak p$ and $Q$, satisfying:
\begin{enumerate}
\item $\deg J\leq \deg\mathfrak p\cdot\deg Q$.
\item $\operatorname{h}(J)\leq\operatorname{h}(\mathfrak p)\cdot\deg Q+\deg\mathfrak p\cdot\operatorname{h}(Q)+C\cdot\deg\mathfrak p\deg Q$.
\item $
    \ln|J(\omega)|
        \leq
    \ln\delta+\deg\mathfrak p\cdot\operatorname{h}(Q)+\operatorname{h}(\mathfrak p)\cdot\deg Q+C\cdot\deg\mathfrak p\deg Q
,$
\end{enumerate}

where
\[
    \delta
        :=
\begin{cases}
    \|Q\|_\omega&\text{if }\;\rho_{\omega,\mathfrak p}<\|Q\|_\omega\\
    |\mathfrak p(\omega)|&\text{if }\;\rho_{\omega,\mathfrak p}\geq \|Q\|_\omega
\end{cases}
\]
\end{Lem}

\begin{Cor}[see \mycite{Chapter 3, Corollary 4.12}{NesterenkoPhilippon2001IntroductionToAlgebraicIndependenceTheory}]\label{Corollary Nesterenko corollary intersection of ideal and polynomial}
Let $S>0$ be a real number and $\eta>0$ be an integer. Assume that $\dim\mathfrak p\geq1$, that $Q\not\in\mathfrak p$ and that:
\begin{enumerate}
\item $\ln|\mathfrak p(\omega)|\leq-S$.
\item $\ln\|Q\|_\omega\leq-2n[K:\mathbb Q]\cdot\deg Q$.
\item $-\eta\ln\|Q\|_\omega\geq2\min(S,\ln\frac1{\rho_{\omega,\mathfrak p}})$.
\end{enumerate}

Then there exists a homogeneous unmixed ideal $J$ of dimension $\dim\mathfrak p -1$, whose zero locus is the intersection of those of $\mathfrak p$ and $Q$, satisfying:
\begin{enumerate}
\item $\deg J\leq\eta\cdot\deg\mathfrak p\cdot\deg Q$.
\item $\operatorname{h}(J)\leq\eta\cdot\left(\operatorname{h}(\mathfrak p)\cdot\deg Q+\deg\mathfrak p\cdot\operatorname{h}(Q)+C\cdot\deg\mathfrak p\deg Q\right)$.
\item $
    \ln|J(\omega)|
        \leq
    -S+\eta\cdot\left(\deg\mathfrak p\cdot\operatorname{h}(Q)+\operatorname{h}(\mathfrak p)\cdot\deg Q+C\cdot\deg\mathfrak p\deg Q\right)
.$
\end{enumerate}
\end{Cor}

\begin{Lem}[see \mycite{Chapter 3, Proposition 4.13}{NesterenkoPhilippon2001IntroductionToAlgebraicIndependenceTheory}]\label{Lemma Nesterenko zero of ideal close to omega}
There exists $x\in\operatorname{V}(I)$ such that
\[
    \deg I\cdot\ln\|\omega-x\|\leq\frac1{1+\dim I}\ln|I(\omega)|+\frac1{1+\dim I}\operatorname{h}(I)+([ K:\mathbb Q]+3)n^3\cdot\deg I
.\]
\end{Lem}

\subsection{Order}

We now introduce our notion of order.

\begin{Def}[Order]\label{Definition order}
Let $f:\Gamma\rightarrow\mathbb C$ be a continuous function that is not identically zero, and let $I\subseteq K[x_0,...,x_n]$ be a homogeneous unmixed ideal. We define the orders of $f$ and $I$ as follows:
\begin{enumerate}
    \item $\operatorname{ord}_{\Gamma}f:=-\ln\max_{\omega\in\Gamma}|f(\omega)|$.
    \item $\operatorname*{ord}_{\Gamma}I:=\operatorname*{ord}_{\Gamma}|I|:=-\ln\max_{\omega\in\Gamma}|I(\omega)|$.
\end{enumerate}
\end{Def}

\section{Sketch of proof of the main theorem}

Our main theorem is \iref{Theorem lower bound main theorem}. \iref{Theorem introduction; main theorem} is obtained as a special case of \iref{Theorem lower bound main theorem} by substituting $I:=(p)$ (see \iref{Corollary the lower bound main theorem; main remark}). Our proof of \iref{Theorem lower bound main theorem} is inspired by Nesterenko's
approach in \cite{nesterenko:mult-nonlinear}. The main idea is to
replace the non-archimedean order of zero at a single point used by
Nesterenko by our archimedean order defined in
\iref{Definition order}. We prove
\iref{Theorem lower bound main theorem} by induction on $\dim I$ as follows:
\begin{enumerate}
\item We express $I$ by its reduced primary decomposition $I=\bigcap_i\mathfrak q_i$, where the radical of each $\mathfrak q_i$ is some prime $\mathfrak p_i$ with multiplicity $n_i$. Then we show that
  \begin{equation}\label{Equation sketch-ord-multlin}
    \operatorname*{ord}_\Gamma I \lesssim \sum n_i \operatorname*{ord}_\Gamma\mathfrak p_i
  ,\end{equation}
  up to some error terms. This allows us to reduce the claim for $I$
  to the claim for each prime component $\mathfrak p_i$; thus, without loss
  of generality, $I$ is prime.
\item We show that one can choose a polynomial $p\in I$ with
  $\operatorname{t}(p) \lesssim \operatorname{t}(I)^\frac{1}{\operatorname{codim} I}$. This uses Nesterenko's estimates
  for Hilbert functions of prime ideals in a fairly standard way.
\item Next, we show that taking $p$ as above with the minimal possible
  $\operatorname{t}(p)$, one of the first $\ell$ derivatives
  $\xi p,\ldots,\xi^\ell p$ does not belong to $I$, where
  $\ell$ is some constant independent of $I$. This is modeled after a
  crucial lemma of Nesterenko \cite{nesterenko:mult-nonlinear}, with
  some modifications needed since we work with $\operatorname{t}(I)$ rather than
  $\deg U$. Denote this derivative by $Q$.
\item If $\operatorname*{ord}_\Gamma I$ is large, then $p$ is identically small on
  $\Gamma$. The Cauchy estimates then imply that $Q$ is identically
  small as well. We deduce that for the ideal $J$, whose locus corresponds to the cycle-theoretic
  intersection of the locus of $I$ with the locus of $Q$, we have
  \begin{equation*}
    \operatorname*{ord}_\Gamma I \lesssim \operatorname*{ord}_\Gamma J.
  \end{equation*}
\item The claim now follows by induction since one easily checks that
  with our choice of $Q$ we have
  \begin{equation*}
    \operatorname{t}(J)^\frac{n}{\operatorname{codim} J} \lesssim \operatorname{t}(I)^\frac{n}{\operatorname{codim} I}.
  \end{equation*}
\end{enumerate}

At various steps in this proof, the more complicated nature of our
$\operatorname*{ord}_\Gamma I$ compared to Nesterenko's local order introduces
technical difficulties. For example, in \iref{Equation sketch-ord-multlin},
since our order is defined by means of the maximum over $\Gamma$,
which may be attained at a different point for each $\mathfrak p_i$, it is not
immediately obvious that this multilinearity should hold. Establishing
the bound requires the use of some value distribution theory of
holomorphic functions to show roughly that the maximum of a product of
holomorphic functions is well approximated by the product of their
maxima. We carry this out using the notion of \emph{Bernstein index};
see the proof of \iref{Proposition I = intersection of pi} for details.

\section{Theorems relating to the Bernstein index}\label{Section bernstein index theorems}

\subsection{Additional lemmas}

In this subsection, we let $f_1,...,f_k$ be a collection of functions holomorphic in $D_1$ and continuous in $\overline{D_1}$.
The following lemma generalizes \iref{Lemma bernstein index lower bound classic} to functions defined as the maximum of holomorphic functions.

\begin{Lem}\label{Lemma Bernstein index lower bound for max of holomorphic functions}
Let $f(z):=\max(|f_1(z)|,...,|f_k(z)|)$. Then, for any $0<h<\frac1e$, there is a collection of discs $\{D^i\}_{i=1}^a$ whose diameters sum to less than $h$ such that
\[
    \forall x\in K-\bigcup_i D^i:|f(x)|\geq\frac{\max_{z\in K}|f(z)|}{e^{C_{K,U}\cdot\ln\frac1h\cdot\B_{K,U}(f)}}
\]
and
\[
    a
        \leq
    C_{K,U}\cdot\B_{K,U}(f)
,\]
where $C_{K,U}$ is a constant depending only on $K,U$.
\end{Lem}

\begin{proof}
Let $i_0$ be an index such that
\[
    \max_{z\in K}|f_{i_0}(z)|=\max_i\max_{z\in K}|f_{i}(z)|=\max_{z\in K}f(z)
.\]
Observe that
\begin{align*}
    \B_{K,U}(f)
    = \ln \frac{\max_{z \in \overline{U}} |f(z)|}{\max_{z \in K} |f(z)|} 
    &= \ln \frac{\max_i \max_{z \in \overline{U}} |f_i(z)|}{\max_{z \in K} |f_{i_0}(z)|} 
    \\&\geq \ln \frac{\max_{z \in \overline{U}} |f_{i_0}(z)|}{\max_{z \in K} |f_{i_0}(z)|}
    = \B_{K,U}(f_{i_0})
.\end{align*}
That is, $\B_{K,U}(f_{i_0})\leq\B_{K,U}(f)$.

By \iref{Lemma bernstein index lower bound classic} and \iref{Remark number of discs}, since $f_{i_0}$ is holomorphic, there is a collection of discs $\{D^i\}_{i=1}^a$ whose diameters sum to less than $h$ such that
\[
    \forall x\in K-\bigcup_i D^i:|f_{i_0}(x)|\geq\frac{\max_{z\in K}|f_{i_0}(z)|}{e^{C_{K,U}\cdot\ln\frac1h\cdot\B_{K,U}(f_{i_0})}}
.\]
and
\[
    a
        \leq
    C_{K,U}\cdot\B_{K,U}(f_{i_0})
        \leq
    C_{K,U}\cdot\B_{K,U}(f)
.\]
Thus, for any such $x\in K-\bigcup_i D^i$,
\begin{align*}
    |f_{i_0}(x)|
    &\geq \frac{\max_{z \in K} |f_{i_0}(z)|}{e^{C_{K,U}\cdot\ln\frac1h\cdot\B_{K,U}(f_{i_0})}} 
    \\&= \frac{\max_{z \in K} |f(z)|}{e^{C_{K,U}\cdot\ln\frac1h\cdot\B_{K,U}(f_{i_0})}}
    \geq \frac{\max_{z \in K} |f(z)|}{e^{C_{K,U}\cdot\ln\frac1h\cdot\B_{K,U}(f)}}
.\end{align*}
Since $|f(x)|\geq|f_{i_0}(x)|$, this completes the proof.
\end{proof}

\begin{Lem}\label{Lemma average is like maximum minus Bernstein index}
Let $f(z):=\max(|f_1(z)|,...,|f_k(z)|)$ and assume that $f\not\equiv 0$. Then, regarding $\overline{D_\frac1e}$ as a probability space with the uniform measure, we have
\[
    \operatorname{\mathbb E}_{z\in\overline{D_\frac1e}}\ln|f(z)|
        \geq
\ln\max_{z\in\overline{D_\frac1e}}|f(z)|-C\cdot\B_{\frac1e,1}(f)
,\]
for some absolute constant $C$ (we define $\ln 0:=-\infty$ and the integral of $-\infty$ over a set of measure $0$ to be 0).
\end{Lem}

\begin{proof}
Consider the corollary of \cite[Lemma 1]{NovikovYakovenko1997AComplexAnalogOfRolle'sTheoremAndPolynomialEnvelopesOfIrreducibleDifferentialEquationsInTheComplexDomain}, which we state below.

\begin{Rem}\label{Remark issue with the lemma}
This Lemma 1 was copied incorrectly from \cite[Lemma 3]{IlyashenkoYakovenko1996CountingRealZerosOfAnalyticFunctionsSatisfyingLinearOrdinaryDifferentialEquations}. The correct statement has an additional multiplicative factor $m=\operatorname*{M}_K(f)=\max_{z\in \overline{D_\frac1e}}|f(z)|$. Thus, Lemma 1 and its corollary are correct only when $\max_{z\in\overline{D_\frac1e}}|f(z)|=1$.
\end{Rem}

In our setting, the corrected version of the corollary implies that for any holomorphic function $h\not\equiv 0$ on $\overline{D_1}$ such that $\max_{z\in\overline{D_\frac1e}}|\operatorname{h}(z)|=1$ and for any $\varepsilon\in\mathbb R_{>0}$, the set $\{t\in \overline{D_\frac1e}:|\operatorname{h}(t)|\leq\varepsilon\}$ can be covered by discs whose diameters sum to at most $\mathrm{const}\cdot\varepsilon^{\frac1{\mathrm{const}\cdot\B_{\frac1e,1}(h)}}$. This also implies that the measure of this set is at most $\mathrm{const}\cdot\varepsilon^{\frac1{\mathrm{const}\cdot\B_{\frac1e,1}(h)}}$.

We wish to apply this corollary to $\hat f:=\frac{f}{\max_{z\in\overline{D_\frac1e}}|f(z)|}$ (we normalized $f$ in accordance with \iref{Remark issue with the lemma}).
We must still justify applying this corollary to our function, which is the maximum of holomorphic functions (divided by a constant). This follows because \iref{Lemma Bernstein index lower bound for max of holomorphic functions} generalizes the underlying lemma to functions such as $\hat f$ that are the maximum of holomorphic functions. Hence the corollary applies to $\hat f$ as well, yielding
\begin{align*}
    \forall\varepsilon>0:&
        \\&\operatorname{\mathbb P}_{z\in\overline{D_\frac1e}}
        \left(
            |\hat f(z)|\leq\varepsilon
        \right)
        \leq
    \mathrm{const}\cdot\varepsilon^{\frac1{\mathrm{const}\cdot\B_{\frac1e,1}\left(\hat f\right)}}
        =
    \mathrm{const}\cdot\varepsilon^{\frac1{\mathrm{const}\cdot\B_{\frac1e,1}(f)}}
,\end{align*}
where $\mathbb P$ denotes probability with respect to the standard probability measure on $\overline{D_\frac1e}$.
Thus, setting $x:=\ln\varepsilon$ gives
\begin{equation}\label{Equation bound on probability}
\forall x\in\mathbb R:\quad
    \operatorname{\mathbb P}_{z\in\overline{D_\frac1e}}
        \left(
            \ln|\hat f(z)|\leq x
        \right)
        \leq
    \mathrm{const}\cdot e^{\frac{x}{\mathrm{const}\cdot\B_{\frac1e,1}(f)}}
.\end{equation}

We now find a lower bound for $\operatorname{\mathbb E}_{z\in\overline{D_\frac1e}}\ln|\hat f(z)|$, where $\mathbb E$ denotes expectation with respect to the standard probability measure on $\overline{D_\frac1e}$.
Observe that $\ln|\hat f|\leq 0$ on $\overline{D_\frac1e}$. We use the following identity for a negative measurable function $h$:
\[
    \operatorname{\mathbb E}_{x\in X} \operatorname{h}(x)
         =
    -\int_{-\infty}^0\operatorname{\mathbb P}_{x\in X}\left(\operatorname{h}(x)\leq y\right)dy
,\]
together with \iref{Equation bound on probability}, to obtain
\begin{align*}
    \operatorname{\mathbb{E}}_{z \in \overline{D_{\frac1e}}} \ln |\hat{f}(z)|
    &= - \int_{-\infty}^0 \operatorname{\mathbb{P}}_{z \in \overline{D_{\frac1e}}} \left( \ln |\hat{f}(z)| \leq x \right) \, dx 
    \geq - \int_{-\infty}^0 \mathrm{const} \cdot e^{\frac{x}{\mathrm{const} \cdot \B_{\frac1e,1}(f)}} \, dx \\
    &= - \mathrm{const} \cdot \B_{\frac1e,1}(f) \cdot e^{\frac{x}{\mathrm{const} \cdot \B_{\frac1e,1}(f)}} \bigg|_{-\infty}^0 
    = - \mathrm{const} \cdot \B_{\frac1e,1}(f)
.\end{align*}
Since $\hat f:=\frac{f}{\max_{z\in\overline{D_\frac1e}}|f(z)|}$, it follows that
\[
    \operatorname{\mathbb E}_{z\in\overline{D_\frac1e}}\ln|f(z)|
        \geq
\ln\max_{z\in\overline{D_\frac1e}}|f(z)|-\mathrm{const}\cdot\B_{\frac1e,1}(f)
.\]
\end{proof}

\begin{Cor}\label{Corollary bernstein index of product}
\[
    \B_{\frac1e,1}(\prod_{i=1}^k f_i)
        \leq
    C\cdot\sum_{i=1}^k \B_{\frac1e,1}(f_i)
\]
for some absolute constant $C$ (independent of $k$).
\end{Cor}

\begin{proof}
We define
\[
    f:=\prod_{i=1}^k f_i
.\]
By normalization, we may assume that
\[
    \forall i:\quad\max_{z\in\overline{D_1}}|f_i(z)|=1 
.\]
This implies that
\[
    \max_{z\in\overline{D_1}}|f(z)|
        \leq
    1
\]
and that
\[
    \forall i:\quad\B_{\frac1e,1}(f_i)=-\max_{z\in\overline{D_\frac1e}}\ln|f(z)|
.\]
Using these bounds, \iref{Lemma average is like maximum minus Bernstein index}, and the fact that the average is $\leq$ the maximum, we obtain
\begin{align*}
    \B_{\frac1e,1}(f)
            &\leq
    -\max_{z\in\overline{D_\frac1e}}\ln|f(z)|
            \leq
    -\mathbb E_{z\in\overline{D_\frac1e}}\ln|f(z)|\\
            &=
    -\sum_i\mathbb E_{z\in\overline{D_\frac1e}}\ln|f_i(z)|
        \leq
    \sum_i
        \mathrm{const}\cdot\B_{\frac1e,1}(f_i)
            -
        \max_{z\in\overline{D_\frac1e}}\ln|f_i(z)|\\
            &=
        (1+\mathrm{const})\sum_i
            \B_{\frac1e,1}(f_i)
.\end{align*}
\end{proof}

\subsection{Derivative bounds}

In this subsection, we prove the following proposition.

\begin{Prop}\label{Proposition derivative-function ratio}
Let $f\not\equiv 0$ be a holomorphic function defined on $D_{1+\varepsilon}$ for some $\varepsilon>0$. Then there exists $A\subseteq\overline{D_\frac1e}$ with $\operatorname{Area}(A)\geq 0.56\cdot\operatorname{Area}\left(\overline{D_\frac1e}\right)$ such that
\[
    \forall z\in A:|f'(z)|\leq C\cdot\B_{\frac1e,1}(f)|f(z)|
.\]
where $C$ is an absolute constant.
\end{Prop}

We begin by proving special cases and an auxiliary lemma (cf. \cite{IlyashenkoYakovenko1996CountingRealZerosOfAnalyticFunctionsSatisfyingLinearOrdinaryDifferentialEquations}).

\subsubsection{The nonzero case}

\begin{Lem}\label{Lemma derivative-function ratio unit}
Let $u\not\equiv 0$ be a holomorphic function defined on $D_{1+\varepsilon}$ for some $\varepsilon>0$. Assume that $u$ has no zeroes in $\overline{D_\frac1e}$. Then
\[
    \forall z\in\overline{D_{0.3}}:|u'(z)|\leq C\cdot\B_{\frac1e,1}(u)|u(z)|
.\]
for some absolute constant $C$.
\end{Lem}

\begin{proof}
Multiplying $u$ by any scalar does not change the desired inequality. Consider
\[
    \max_{z\in\overline{D_\frac1e}}|u(z)|
.\]
Let $z_0\in\overline{D_\frac1e}$ be a point where this maximum is attained. By the maximum principle for holomorphic functions, we can choose $z_0$ to lie on the boundary of $\overline{D_\frac1e}$. Thus, by multiplying $u$ by a scalar, we may assume that
\[
    \max_{z\in\overline{D_\frac1e}}|u(z)|=|u(z_0)|,u(z_0)=1
.\]

Define another holomorphic function:
\[
\begin{cases}
    g:\overline{D_\frac1e}\rightarrow\mathbb C
    \\g(z):=\ln u(z)
\end{cases}
\]
by choosing $g(z_0)$ to be $0$ (we lift $u(\overline{D_\frac1e})$ to the universal cover of $\mathbb C-\{0\}$ and then apply $\ln$ on the universal cover, where it is well defined. Thus $g(z)$ may have a large imaginary component).

We now seek upper bounds for $|g(z)|$.

\paragraph{Upper bound $|\operatorname{re}(g(z))|\leq\mathrm{const}\cdot\B_{\frac1e,1}(u)$ on some circle centered at $0$ of radius at least $0.35$:}
\[
    \forall z\in\mathbb C-\{0\}:\operatorname{re}(\ln(z))=\ln|z|
.\]
Hence,
\begin{equation}\label{Equation real part of ln}
    |\operatorname{re}(g(z))|=|\ln |u(z)\|
.\end{equation}
Thus, by the behaviour of $\ln$, to bound $|\operatorname{re}(g(z))|$ from above, we need bounds for $\min|u(z)|$ and $\max|u(z)|$.
\begin{enumerate}
    \item $\max|u(z)|$ over $\overline{D_\frac1e}$:
By our earlier normalization,
\[
    \max_{z\in\overline{D_\frac1e}}|u(z)|=1
.\]
\item $\min|u(z)|$ over some circle centered at $0$ of radius at least $0.35$:
By \iref{Lemma bernstein index lower bound classic}, with $h:=0.01$ and $f:=u$, there is a finite collection of discs $\{D^i\}_{i=1}^a$ whose diameters sum to less than $0.01$ such that
\[
    \forall z\in\overline{D_\frac1e}-\bigcup_i D^i:\quad
    |u(z)|
        \geq
    \left(\frac1e\right)^{\mathrm{const}\cdot\B_{\frac1e,1}(u)}
.\]
Since the sum of the diameters of the $D^i$ is less than $0.01$, there exists some $0.35\leq r_0\leq \frac1e$ such that the circle $\partial\overline{D_{r_0}}$, centered at the origin, does not intersect any of the $D^i$. Therefore,
\[
    \min_{z\in \partial\overline{D_{r_0}}}|u(z)|\geq\left(\frac1e\right)^{\mathrm{const}\cdot\B_{\frac1e,1}(u)}
.\]
\end{enumerate}

We have therefore found some $0.35\leq r_0\leq\frac1e$ such that
\[
    \forall z\in \partial\overline{D_{r_0}}:\left(\frac1e\right)^{\mathrm{const}\cdot\B_{\frac1e,1}(u)}\leq|u(z)|\leq 1
.\]
Hence, by \iref{Equation real part of ln},
\[
    \forall z\in \partial\overline{D_{r_0}}:|\operatorname{re}(g(z))|\leq \mathrm{const}\cdot\B_{\frac1e,1}(u)
.\]

\paragraph{Upper bound $|\operatorname{im}(g(z))|\leq \mathrm{const}\cdot\B_{\frac1e,1}(u)$ on $\overline{D_\frac1e}$:}

Since $\operatorname{im}(g(z))$ is harmonic on $\overline{D_\frac1e}$, the maximum principle for harmonic functions implies that it attains its maximum and minimum on $\partial\overline{D_\frac1e}$. Therefore, $|\operatorname{im}(g(z))|$ attains its maximum on $\partial\overline{D_\frac1e}$. Recall that we fixed a point $z_0\in \partial\overline{D_\frac1e}$ and defined $g(z_0)$ to be $0$. We have
\[
    \forall z\in \partial\overline{D_\frac1e}:\operatorname{im}(g(z))=\int_{\partial\overline{D_\frac1e}:z_0\rightarrow z}\frac{d}{dt}\arg(u(t))dt
\]
(where the integral is taken over the counterclockwise path from $z_0$ to $z$ along $\partial\overline{D_\frac1e}$).
Hence,
\[
    \forall z\in \partial\overline{D_\frac1e}:|\operatorname{im}(g(z))|\leq\int_{\partial\overline{D_\frac1e}:z_0\rightarrow z}\left|\frac{d}{dt}\arg(u(t))dt\right|\leq\oint_{\partial\overline{D_\frac1e}}\left|\frac{d}{dt}\arg(u(t))dt\right|
.\]

The quantity $\frac1{2\pi}\int_{\partial\overline{D_\frac1e}}\left|\frac{d}{dt}\arg(u(t))dt\right|$ is the Voorhoeve index of $u$ along $\partial\overline{D_\frac1e}$, $\operatorname{\mathbf{V}}_{\partial\overline{D_\frac1e}}(u)$ (see \cite{KhovanskiiYakovenko1996GeneralizedRolleTheoremInRnAndC}), and \cite[Theorem 3]{KhovanskiiYakovenko1996GeneralizedRolleTheoremInRnAndC} gives
\[
    \operatorname{\mathbf{V}}_\Gamma(u)\leq\alpha(K,U,\Gamma)\B_{K,U}(u)
,\]
where $\alpha(K,U,\Gamma)$ is a constant depending only on $K,U,\Gamma$.
Substituting $K=\overline{D_\frac1e},U=D_1,\Gamma=\partial\overline{D_\frac1e}$ gives
\[
    \oint_{\partial\overline{D_\frac1e}}\left|\frac{d}{dt}\arg(u(t))dt\right|=2\pi\operatorname{\mathbf{V}}_{\partial\overline{D_\frac1e}}(u)\leq \mathrm{const}\cdot\B_{\frac1e,1}(u)
.\]
Hence,
\[
    \forall z\in \partial\overline{D_\frac1e}:|\operatorname{im}(g(z))|\leq \mathrm{const}\cdot\B_{\frac1e,1}(u)
,\]
and by the maximum (and minimum) principle applied to the harmonic function $\operatorname{im}\circ\,g$,
\[
    \forall z\in\overline{D_\frac1e}:|\operatorname{im}(g(z))|\leq \mathrm{const}\cdot\B_{\frac1e,1}(u)
.\]

\paragraph{Conclusion:}

Combining the bounds for the real and imaginary parts, we find that for some $0.35\leq r_0\leq\frac1e$,
\[
    \forall z\in \partial\overline{D_{r_0}}:|g(z)|\leq \mathrm{const}\cdot\B_{\frac1e,1}(u)
.\]

Cauchy's differentiation formula gives
\[
    \forall z\in\overline{D_{0.3}}:g'(z)=\frac1{2\pi i}\oint_{\partial\overline{D_{r_0}}}\frac{g(t)}{(t-z)^2}dt
.\]
Hence,
\[
    \forall z\in\overline{D_{0.3}}:|g'(z)|\leq\oint_{\partial\overline{D_{r_0}}}\frac{|g(t)|}{|t-z|^2}dt
.\]
Since $|t-z|\geq 0.05$ and $|g(t)|\leq \mathrm{const}\cdot\B_{\frac1e,1}(u)$, we obtain
\[
    \forall z\in\overline{D_{0.3}}:|g'(z)|\leq \mathrm{const}\cdot\B_{\frac1e,1}(u)
.\]
Since $g'=\left(\ln(u)\right)'=\frac{u'}{u}$, it follows that
\[
    \forall z\in\overline{D_{0.3}}:|u'(z)|\leq \mathrm{const}\cdot\B_{\frac1e,1}(u)|u(z)|
.\]
\end{proof}

\subsubsection{The polynomial case}

\begin{Lem}\label{Lemma derivative-function ratio polynomial}
Let $p(z)=\prod_i(z-z_i)$ be a polynomial where all the ${z_i}$ are inside $\overline{D_r}$ for some $r>0$ (the $z_i$ may repeat). Then there is a subset of $\overline{D_r}$ of measure at least $0.9\operatorname{Area}\left(\overline{D_r}\right)$ on which
\[
    |p'(z)|\leq\frac{40}{r}\deg p\cdot|p(z)|
.\]
\end{Lem}

\begin{proof}
Observe that
\[
    \frac{p'}{p}(z)=\sum_{i=1}^{\deg p}\frac1{z-z_i}
.\]
Thus,
\[
    0\leq\left|\frac{p'}{p}(z)\right|\leq\sum_{i=1}^{\deg p}\frac1{|z-z_i|}
.\]
Regarding $\overline{D_r}$ as a probability space with the uniform measure, we obtain
\[
    \operatorname{\mathbb E}_{z\in \overline{D_r}}\left|\frac{p'}{p}(z)\right|\leq\sum_{i=1}^{\deg p}\operatorname{\mathbb E}_{z\in \overline{D_r}}\frac1{|z-z_i|}
.\]

We now derive an upper bound for $\operatorname{\mathbb E}_{z\in \overline{D_r}}\frac1{|z-z_i|}$:
\[
    \operatorname{\mathbb E}_{z\in \overline{D_r}}\frac1{|z-z_i|}=\frac1{\pi r^2}\int_{z\in \overline{D_r}}\frac1{|z-z_i|}dA
,\]
where $dA$ is the Euclidean area differential $2$-form of $\mathbb R^2$.
Consider a disc of radius $2r$ centered at $z_i$, denoted by $\overline{D_{z_i,2r}}$. Since $\overline{D_{z_i,2r}}\supseteq \overline{D_r}$, we obtain
\begin{align*}
    \frac1{\pi r^2}\int_{z\in \overline{D_r}}\frac1{|z-z_i|}dA
        &\leq
    \frac1{\pi r^2}\int_{z\in \overline{D_{z_i,2r}}}\frac1{|z-z_i|}dA
        \\&=
    \frac1{\pi r^2}\int_{z\in\overline{D_{2r}}}\frac1{|z|}dA
        =
    \frac4{r}
.\end{align*}
It follows that
\[
    \operatorname{\mathbb E}_{z\in \overline{D_r}}\left|\frac{p'}{p}(z)\right|\leq\frac{4\deg p}{r}
.\]
Thus, by Markov's inequality,
\[
    \operatorname{\mathbb P}_{z\in \overline{D_r}}\left(\left|\frac{p'}{p}(z)\right|>\frac{40\deg p}{r}\right)\leq 0.1
.\]
Thus, there is a subset of $\overline{D_r}$ of measure at least $0.9\operatorname{Area}\left(\overline{D_r}\right)$ on which
\[
|p'(z)|\leq\frac{40}{r}\deg p\cdot|p(z)|
.\]
\end{proof}

\subsubsection{The general case}

We begin with the following auxiliary lemma.

\begin{Lem}\label{Lemma Bernstein index of product with polynomial}
Let $p(z)=\prod_i(z-z_i)$ be a polynomial where all the ${z_i}$ are inside $\overline{D_\frac1e}$ (the $z_i$ may repeat), and let $g\not\equiv 0$ be a holomorphic function defined on $D_{1+\varepsilon}$ for some $\varepsilon>0$.
Define $f:=g\cdot p$. Then
\[
    \B_{\frac1e,1}(f)\geq\B_{\frac1e,1}(g)-\ln(\tfrac{2}{e-1})\deg p
.\]
\end{Lem}

\begin{proof}
Recall that
\[
    \B_{\frac1e,1}(f):=\ln\frac{\max_{z\in\overline{D_1}}|f(z)|}{\max_{z\in\overline{D_\frac1e}}|f(z)|}
.\]
By the maximum principle for holomorphic functions, these maxima are attained on the boundaries $\partial\overline{D_\frac1e},\partial\overline{D_1}$ of $\overline{D_\frac1e},\overline{D_1}$:
\[
    \B_{\frac1e,1}(f):=\ln\frac{\max_{z\in \partial\overline{D_1}}|f(z)|}{\max_{z\in \partial\overline{D_\frac1e}}|f(z)|}=\ln\frac{\max_{z\in \partial\overline{D_1}}|g(z)\|p(z)|}{\max_{z\in \partial\overline{D_\frac1e}}|g(z)\|p(z)|}
.\]

When $z\in \partial\overline{D_1}$, we have $|z-z_i|\geq1-\frac1e$, so $|p(z)|\geq\left(1-\frac1e\right)^{\deg p}$. When $z\in \partial\overline{D_\frac1e}$, we have $|z-z_i|\leq\frac2e$, so $|p(z)|\leq(\frac2e)^{\operatorname{deg}p}$.
Thus,
\[
    \max_{z\in \partial\overline{D_1}}|g(z)\|p(z)|\geq\left(1-\frac1e\right)^{\deg p}\max_{z\in \partial\overline{D_1}}|g(z)|
\]
and
\[
    \max_{z\in \partial\overline{D_\frac1e}}|g(z)\|p(z)|\leq\left(\frac2e\right)^{\operatorname{deg}p}\max_{z\in \partial\overline{D_\frac1e}}|g(z)|
.\]
\enlargethispage{\baselineskip}
Combining these bounds gives
\[
    \B_{\frac1e,1}(f)\geq\ln\frac{\left(1-\frac1e\right)^{\deg p}\max_{z\in \partial\overline{D_1}}|g(z)|}{(\frac2e)^{\operatorname{deg}p}\max_{z\in \partial\overline{D_\frac1e}}|g(z)|}\geq\B_{\frac1e,1}(g)-\ln(\tfrac{2}{e-1})\deg p
.\]
\end{proof}

We now prove the general case.

\begin{proof}[Proof of \iref{Proposition derivative-function ratio}]
Consider the zeroes of $f$ in $\overline{D_\frac1e}$. Define
\[
    p(z):=\prod_i(z-z_i)
\]
to be the product of these zeroes, with repetitions accounting for their multiplicities.
Define $u:=\frac{f}{p}$, or equivalently, $f:=up$. Then $u$ is a holomorphic function on $D_{1+\varepsilon}$ without zeroes in $\overline{D_\frac1e}$.

We use the identity
    \[
    \frac{f'}{f}=\frac{u'}{u}+\frac{p'}{p}
    .\]
To bound $\left|\frac{f'}{f}\right|$ from above, we use
\[
    \left|\frac{f'}{f}\right|\leq\left|\frac{u'}{u}\right|+\left|\frac{p'}{p}\right|
.\]

By \iref{Lemma derivative-function ratio unit},
\[
    \forall z\in\overline{D_{0.3}}:\left|\frac{u'}{u}\right|(z)\leq \mathrm{const}\cdot\B_{\frac1e,1}(u)
,\]
and by \iref{Lemma derivative-function ratio polynomial}, there is a subset $\tilde A$ of $\overline{D_\frac1e}$ with $\operatorname{Area}(\tilde A)\geq 0.9\operatorname{Area}\left(\overline{D_\frac1e}\right)$ on which
\[
    \forall z\in\tilde A:\left|\frac{p'}{p}\right|(z)\leq \mathrm{const}\cdot\deg p
.\]
Combining these bounds and defining $A:=\overline{D_{0.3}}\cap\tilde A$, we obtain
\[
    \operatorname{Area}(A)\geq 0.56\operatorname{Area}\left(\overline{D_\frac1e}\right)
\]
and
\[
    \forall z\in A:\left|\frac{f'}{f}\right|(z)\leq \mathrm{const}\cdot(\B_{\frac1e,1}(u)+\deg p)
.\]

To complete the proof, it remains to show that $\B_{\frac1e,1}(u)\leq \mathrm{const}\cdot\B_{\frac1e,1}(f)$ and that $\deg p\leq \mathrm{const}\cdot\B_{\frac1e,1}(f)$.
The bound $\deg p\leq \mathrm{const}\cdot\B_{\frac1e,1}(f)$ follows directly from \iref{Lemma bound on number of zeroes of f given bernstein index}, since $f$ has $\deg p$ zeroes on $\overline{D_\frac1e}$. For the other bound, \iref{Lemma Bernstein index of product with polynomial} gives
\[
    \B_{\frac1e,1}(f)\geq\B_{\frac1e,1}(u)-\deg p
.\]
Thus,
\[
    \B_{\frac1e,1}(u)\leq\B_{\frac1e,1}(f)+\deg p
    \leq \mathrm{const}\cdot\B_{\frac1e,1}(f)
,\]
which completes the proof.
\end{proof}

\section{Finding low-\texorpdfstring{$\operatorname{t}$}{t} polynomials with low order in an ideal}

Consider the result in \cite[Lemma 5.4]{Nesterenko1996ModularFunctionsAndTranscendenceQuestions}, which was presented again in \cite[Lemmas 12 and 13]{Binyamini2014MultiplicityEstimatesAnalyticCyclesAndNewtonPolytopesV2}:

\begin{Prop}[see \mycite{Lemmas 12 and 13}{Binyamini2014MultiplicityEstimatesAnalyticCyclesAndNewtonPolytopesV2}]
Let $\mathfrak p\subseteq\mathbb C[x_0,...,x_n]$ be a homogeneous prime ideal that does not contain any nonzero $\{Q,\xi(Q),\xi(\xi(Q)),...\}$ and let $P\in\mathfrak p-\{0\}$ be of minimal $\deg$. Then there exists a constant $C$ depending only on $n,\xi$ such that
\[
    \deg P
        \leq
    C\cdot(\deg\mathfrak p)^{\frac1{\operatorname{codim}\mathfrak p}}
\]
and
\[
    \xi^i(P)\not\in\mathfrak p
,\]
for some $i$-th derivative $\xi^i(P)$ with $i\leq C$.
\end{Prop}

\begin{Rem}
In \cite{Binyamini2014MultiplicityEstimatesAnalyticCyclesAndNewtonPolytopesV2}, the author works with varieties and vector fields in affine space $\mathbb C^n$. The proofs also work in the projective case $\mathbb C^{n+1}-\{0\}$, so we describe the theorems and definitions from \cite{Binyamini2014MultiplicityEstimatesAnalyticCyclesAndNewtonPolytopesV2} over projective space for consistency with the rest of this article.
\end{Rem}

This theorem is a common ingredient in proofs of results similar to our main theorem. Our goal in this section is to modify its proof as presented in \cite{Binyamini2014MultiplicityEstimatesAnalyticCyclesAndNewtonPolytopesV2} to apply to ideals in $K[x_0,...,x_n]$ instead of $\mathbb C[x_0,...,x_n]$, using our function $\operatorname{t}=\deg+\operatorname{h}$ instead of $\deg$. We start by replacing \cite[Lemma 12]{Binyamini2014MultiplicityEstimatesAnalyticCyclesAndNewtonPolytopesV2} with a generalization of \cite[Corollary 3]{Nesterenko1985EstimatesForTheCharacteristicFunctionOfAPrimeIdeal}, and then explain the changes needed in the proof of \cite[Lemma 13]{Binyamini2014MultiplicityEstimatesAnalyticCyclesAndNewtonPolytopesV2}. The modified proposition is as follows:

\begin{PropThm}\label{Proposition existence of polynomial of small t and small mult in ideal}
Let $\mathfrak p\subseteq K[x_0,...,x_n]$ be a homogeneous prime ideal that does not contain any nonzero $\{Q,\xi(Q),\xi(\xi(Q)),...\}$ and let $P\in\mathfrak p-\{0\}$ be of minimal $\operatorname{t}$. Then there exists a constant $C$ depending only on $n,\xi,K$ such that
\[
    \operatorname{t}(P)
        \leq
    C\cdot\operatorname{t}(\mathfrak p^\mathrm{proj})^{\frac1{\operatorname{codim}\mathfrak p}}
\]
and
\[
    \xi^i(P)\not\in\mathfrak p
,\]
for some $i$-th derivative $\xi^i(P)$ with $i\leq C$.
\end{PropThm}

\subsection{Replacing \texorpdfstring{\cite[Lemma 12]{Binyamini2014MultiplicityEstimatesAnalyticCyclesAndNewtonPolytopesV2}}{the first lemma}}

In this subsection, we prove the following lemma:

\begin{Lem}\label{Lemma existence of polynomial of small t in ideal}
Let $\mathfrak p\subseteq K[x_0,...,x_n]$ be any (nontrivial) homogeneous prime ideal. Then there exists some $P\in\mathfrak p-\{0\}$ such that:
\[
    \operatorname{t}(P)
        \leq
    C\cdot\operatorname{t}(\mathfrak p)^{\frac1{\operatorname{codim}(\mathfrak p)}}
,\]
where $C$ is some constant depending only on $K,n$.
\end{Lem}

A straightforward computation gives the following corollary, which replaces \cite[Lemma 12]{Binyamini2014MultiplicityEstimatesAnalyticCyclesAndNewtonPolytopesV2}:

\begin{Cor}\label{Corollary existence of polynomial of specified t in ideal with small t}
There exists a constant $C$ depending only on $K,n$ such that for any $t$ and any (nontrivial) homogeneous prime ideal $\mathfrak p$ satisfying
\[
    \operatorname{t}(\mathfrak p)\leq\frac1C\cdot t^{\operatorname{codim}(\mathfrak p)}
,\]
there exists some $P\in\mathfrak p-\{0\}$ with
\[
    \operatorname{t}(P)<t
.\]
\end{Cor}

We note that \iref{Lemma existence of polynomial of small t in ideal} is a simple generalization of the following lemma:

\begin{Lem}[\mycite{Corollary 3}{Nesterenko1985EstimatesForTheCharacteristicFunctionOfAPrimeIdeal}]\label{Lemma every ideal over Z has low t polynomial}
Let $\mathfrak q\subseteq\mathbb Z[x_0,...,x_n]$ (with $n\geq 1$) be a homogeneous prime ideal for which $\mathfrak q\cap\mathbb Z=\{0\}$. Then there exists some homogeneous $Q\in\mathfrak q-\{0\}$ such that:
\[
    \operatorname{t}(Q)
        \leq
    C\cdot\operatorname{t}(\mathfrak q)^{\frac1{\operatorname{codim}(\mathfrak q)}}
,\]

where $C$ is some constant depending only on $n$.
\end{Lem}

We now use this lemma to establish \iref{Lemma existence of polynomial of small t in ideal}.

\begin{proof}[Proof of \iref{Lemma existence of polynomial of small t in ideal}]
Given $\mathfrak p$, our goal is to find an ideal $\mathfrak q\subseteq\mathbb Z[x_0,...,x_n]$ contained in $\mathfrak p$ whose $\operatorname{t}$ is comparable to that of $\mathfrak p$, and then apply \iref{Lemma every ideal over Z has low t polynomial} to it.

Let
\[
    \mathfrak p^\mathbb Q
        :=
    \bigcap_{\sigma\in\operatorname{Gal}(K,\mathbb Q)}\mathfrak\sigma(\mathfrak p)
\]

(note that this is the reduced primary decomposition of $\mathfrak p^\mathbb Q$),
and let
\[
    \mathfrak p^{\mathbb Z}
        :=
    \mathfrak p^\mathbb Q\cap\mathbb Z[x_1,...,x_n]
\]
($\mathfrak p^\mathbb Q$ is generated by elements of $\mathbb Z[x_1,...,x_n]$ since $\mathfrak p^\mathbb Q$ is invariant under $\operatorname{Gal(K,\mathbb Q)}$).

Up to a scalar, the Chow form of $\mathfrak p^\mathbb{Z}$ is given by
\[
\operatorname{Chow}(\mathfrak p^{\mathbb Z})=\operatorname{Chow}(\mathfrak p^\mathbb Q)=\prod_{\sigma\in\operatorname{Gal}(K,\mathbb Q)}\operatorname{Chow}(\sigma(\mathfrak p))
.\]
This implies that
\begin{equation}\label{Equation bound on t(p^Z)}
    \operatorname{t}(\mathfrak p^{\mathbb Z})
        \leq
    \mathrm{const}\cdot\operatorname{t}(\mathfrak p)
.\end{equation}
Thus, the desired polynomial is obtained by applying \iref{Lemma every ideal over Z has low t polynomial} to $\mathfrak p^\mathbb Z$ (note that $\mathfrak p^{\mathbb Z}\cap\mathbb Z=0$).
\end{proof}

\subsection{Modifying the proof of \texorpdfstring{\cite[Lemma 13]{Binyamini2014MultiplicityEstimatesAnalyticCyclesAndNewtonPolytopesV2}}{the second lemma}}

In this subsection, our goal is to obtain the following result:

\begin{Prop}\label{Proposition polynomial with minimal t has low multiplicity}
Let $\mathfrak p\subseteq K[x_0,...,x_n]$ be a homogeneous prime ideal that does not contain any nonzero $\{Q,\xi(Q),\xi(\xi(Q)),...\}$ and let $P\in\mathfrak p-\{0\}$ be of minimal $\operatorname{t}$. Then there exists a constant $C$ depending only on $n,\xi,K$ such that
\[
    \xi^i(P)\not\in\mathfrak p
,\]
for some $i$-th derivative $\xi^i(P)$ with $i\leq C$.
\end{Prop}

This proposition, together with \iref{Lemma existence of polynomial of small t in ideal}, suffices to establish \iref{Proposition existence of polynomial of small t and small mult in ideal}.
This proposition is nearly the same as \cite[Lemma 13]{Binyamini2014MultiplicityEstimatesAnalyticCyclesAndNewtonPolytopesV2}, but concerns ideals over $K$ instead of $\mathbb C$ and uses $\operatorname{t}$ instead of $\deg$.

The proof follows \cite[Lemma 13]{Binyamini2014MultiplicityEstimatesAnalyticCyclesAndNewtonPolytopesV2} with minor modifications, which we describe below. The reader should compare these modifications with the proof in \cite[Lemma 13]{Binyamini2014MultiplicityEstimatesAnalyticCyclesAndNewtonPolytopesV2} and be familiar with the notions in \cite[\S2, 'cycles and multiplicities']{Binyamini2014MultiplicityEstimatesAnalyticCyclesAndNewtonPolytopesV2}.

\subsubsection{Outline}

The following subsubsections describe the changes needed in the proof of \cite[Lemma 13]{Binyamini2014MultiplicityEstimatesAnalyticCyclesAndNewtonPolytopesV2}.
\begin{enumerate}
    \item We start by modifying the definitions of cycles and multiplicities to be over $K$ instead of $\mathbb C$.
    \item We explain how to use these definitions in the proof and why the auxiliary results \cite[Proposition 10 and Lemma 8]{Binyamini2014MultiplicityEstimatesAnalyticCyclesAndNewtonPolytopesV2} apply in this setting.
    \item Finally, we explain how to replace the $\deg$ used in the proof with our notion of $\operatorname{t}$.
\end{enumerate}

Together, these changes yield the proof of \iref{Proposition polynomial with minimal t has low multiplicity}.

\subsubsection{Useful notions and modified definitions of cycles and multiplicities}

\begin{Def}
We say that a variety $V\subseteq\mathbb C^n$ is irreducible over $K$ if it is defined over $K$ and its corresponding ideal over $K$ is prime.
\end{Def}

\begin{Rem}\label{Remark decomposition of prime ideal over K}
Let $I$ be an ideal over $K$ and let $I^{\overline{\mathbb Q}}$ be the ideal that it generates over $\overline{\mathbb Q}$. Then, $I$ is prime if and only if there is a prime ideal $J$ over $\overline{\mathbb Q}$ such that $I$ has a prime decomposition with components $I^{\overline{\mathbb Q}}=\bigcap\left\{\sigma(J):\sigma\in\operatorname{Gal(\overline{\mathbb Q},K)}\right\}$ (note that this set is finite).

Analogously, a variety $V$ will be irreducible over $K$ if and only if it can be expressed as a union 
\[
V=\bigcup_{\sigma\in\operatorname{Gal(\overline{\mathbb Q},K)}}\sigma(W)
,\]
for some variety $W$ defined over $\overline{\mathbb Q}$ that is irreducible over $\mathbb C$ (or equivalently over $\overline{\mathbb Q}$). This is equivalent to $[V]$ being the sum with multiplicities $1$
\[
    [V]=\sum_i[V_i]
,\]
where $V_i$ are the conjugates $\sigma(W)$, which are defined over $\overline{\mathbb Q}$ and irreducible over $\mathbb C$.
\end{Rem}

\begin{Def}[Cycle over $K$]
We say that a cycle $\Omega$ is defined over $K$ if $\Omega$ can be written as a sum $\Omega=\sum_i m_i[V_i]$, where each $V_i$ is defined over $K$.
\end{Def}

\begin{Def}[Multiplicity over $K$]
Let $V,W$ be varieties defined over $K$, such that $V$ is irreducible over $K$. Write $[V]=[V_1]+...+[V_k]$ and $[W]=[W_1]+...+[W_l]$ for the decompositions given in \iref{Remark decomposition of prime ideal over K} (so that the $V_i$ are conjugate under $\operatorname{Gal}(\overline{\mathbb Q},K)$), and let $p\in\mathbb C^{n+1}-\{0\}$. Assume that $\xi$ is also defined over $K$. Then we define:
\[
    \operatorname{mult}_p W:=\sum_j\operatorname{mult}_p W_j
\]
and
\[
    \operatorname{mult}_p^\xi W:=\sum_j\operatorname{mult}_p^\xi W_j
,\]
and we define
\[
    \operatorname{mult}_V W:=\sum_j\operatorname{mult}_{V_1}W_j
\]
and
\[
    \operatorname{mult}_V^\xi W:=\sum_j\operatorname{mult}_{V_1}^\xi W_j
\]
where this last equality holds since the $V_i$ are conjugate to each other, so $\operatorname{mult}_{V_i}W$ and $\operatorname{mult}_{V_i}^\xi W$ do not depend on $i$. Here $\operatorname{mult}_{V_i}^\xi W_j$ is understood in the usual sense of \cite[\S2.4, 'multiplicities at generic points']{Binyamini2014MultiplicityEstimatesAnalyticCyclesAndNewtonPolytopesV2}.
\end{Def}

\begin{Rem}
These definitions agree with those of \cite{Binyamini2014MultiplicityEstimatesAnalyticCyclesAndNewtonPolytopesV2} if $V$ is also irreducible over $\mathbb C$ and $W$ is regarded as a cycle over $\mathbb C$. In any case, $\operatorname{mult}_V W$ equals the value of $\operatorname{mult}_p W$ for a generic $p\in V$. The same holds for $\operatorname{mult}_V^\xi W$.

In addition, for $V$ irreducible over $K$, for $\xi$ defined over $K$, and for any polynomial $P\in K[x_0,...,x_n]$, the multiplicity $\operatorname{mult}_V^\xi P$ equals the number of times we need to differentiate $P$ along $\xi$ to obtain a polynomial that does not belong to $I_V$, as in the case where $V$ is irreducible over $\mathbb C$ in \cite{Binyamini2014MultiplicityEstimatesAnalyticCyclesAndNewtonPolytopesV2}.
\end{Rem}

\begin{Def}
Let $V$ and $W$ be varieties defined over $K$. We write $[V]=[V_1]+...+[V_k]$ and $[W]=[W_1]+...+[W_l]$ as in \iref{Remark decomposition of prime ideal over K}. We define
\[
    V\cdot W:=\sum_{i,j}[V_i\cdot W_j]
\]
where $V_i\cdot W_j$ is computed as in \cite{Binyamini2014MultiplicityEstimatesAnalyticCyclesAndNewtonPolytopesV2}. This definition treats $V,W$ as cycles over $\mathbb C$ and also extends to cycles defined over $K$.
\end{Def}

\begin{Rem}\label{Remark Product of cycle over K with function}
If $V$ is an irreducible variety over $K$ of dimension $\geq 1$ and $W$ is the zero locus of some polynomial $P\in K[x_0,...,x_n]$ (assuming that $W$ does not contain $V$), the construction of $[V]\cdot \operatorname{V}(P)$ yields the same decomposition as the construction in \iref{Lemma Nesterenko proposition intersection of ideal and polynomial} applied to the ideal corresponding to $V$ over $K$ and the polynomial $P$. Thus, by \iref{Lemma Nesterenko proposition intersection of ideal and polynomial} and \iref{Lemma Nesterenko principal ideal}, we have:
\begin{itemize}
\item $[V]\cdot \operatorname{V}(P)$ is defined over $K$.
\item $\deg([V]\cdot \operatorname{V}(P))\leq\deg V\cdot\deg P$.
\item $\operatorname{h}([V]\cdot \operatorname{V}(P))\leq\operatorname{h}(V)\cdot\deg P+\operatorname{h}(P)\cdot\deg V+C\cdot\deg V\cdot\deg P$, where $C$ is some constant depending only on $K,n$.
\end{itemize}
\end{Rem}

\subsubsection{The first change: working over \texorpdfstring{$K$}{K}}

Let $V\subseteq\mathbb C^{n+1}-\{0\}$ be the zero locus of $\mathfrak p$.
The main idea is to replace 'irreducible varieties' in the original proof with 'irreducible varieties over $K$' and 'cycles' with 'cycles defined over $K$' (so that we also work only with varieties defined over $K$). We start by checking that all the objects in the proof are defined over $K$.

$\Gamma^j$ from the proof will be defined over $K$ and written as $\Gamma^j=\sum_i m_i^j[W_i^j]$, where $W_i^j$ are irreducible over $K$. Because $\xi,P$ are defined over $K$, so is $Q_i^j$. Now, since
\[
    \tilde{\Gamma}_i^{j+1}:=[W_i^j]\cdot \operatorname{V}(Q_i^j):=\sum_k [(W_i^j)_k]\cdot \operatorname{V}(Q_i^j)
\]
(where $[W_i^j]=\sum_k[(W_i^j)_k]$ is the decomposition from \iref{Remark decomposition of prime ideal over K}), by \iref{Remark Product of cycle over K with function}, $\tilde{\Gamma}_i^{j+1}$ will be defined over $K$, as will $\Gamma_i^{j+1}$ (it is defined as the components of $\Gamma_i^{j+1}$, irreducible over $K$, which contain $V$).

Finally, we explain why the auxiliary results used in the original proof extend to varieties and cycles defined over $K$.
\begin{itemize}
\item \cite[Theorem 2]{Binyamini2014MultiplicityEstimatesAnalyticCyclesAndNewtonPolytopesV2} and \cite[Proposition 6]{Binyamini2014MultiplicityEstimatesAnalyticCyclesAndNewtonPolytopesV2} can still be used without change.
\item The generalization of \cite[Lemma 12]{Binyamini2014MultiplicityEstimatesAnalyticCyclesAndNewtonPolytopesV2} to ideals over $K$ follows similarly to the proof of \iref{Corollary existence of polynomial of specified t in ideal with small t}. We omit the details, since we replace this lemma with \iref{Corollary existence of polynomial of specified t in ideal with small t} when replacing $\deg$ with $\operatorname{t}$ in \iref{Subsubsection the second change}. This gives a constant different from $A_n$.
\item  \cite[Proposition 10]{Binyamini2014MultiplicityEstimatesAnalyticCyclesAndNewtonPolytopesV2} works for $V$ irreducible over $K$ as well since $
    \operatorname{mult}_V^\xi V
        :=
    \operatorname{mult}_{V_i}^\xi V
        =
    \sum_j \operatorname{mult}_{V_i}^\xi V_j
        =
    \operatorname{mult}_{V_i}^\xi V_i$ (if $i\neq j$ then $V_i\not\subseteq V_j$ and so $\operatorname{mult}_{V_i}^\xi V_j=0$), where $V_i$ are the components as in \iref{Remark decomposition of prime ideal over K}, and since $V$ is invariant under $\xi$ if and only if each/any $V_i$ is (assuming $\xi$ is defined over $K$ as well).
    \item \cite[Proposition 4]{Binyamini2014MultiplicityEstimatesAnalyticCyclesAndNewtonPolytopesV2} applies directly to any cycle $\Gamma$ defined over $K$, using our modified definitions.
    \item \cite[Lemma 8]{Binyamini2014MultiplicityEstimatesAnalyticCyclesAndNewtonPolytopesV2} works for $V$ irreducible over $K$ (assuming $f,\xi$ are defined over $K$ as well) because $f\in I_V$ if and only if $f\in I_{V_i}$ for each/any $V_i$ (and similarly for $g$) where $V_i$ are the components as in \iref{Remark decomposition of prime ideal over K}, and thus the original inequality extends to our case that is over $K$ by linearity.
\end{itemize}

\subsubsection{The second change: using \texorpdfstring{$\operatorname{t}$}{t} instead of \texorpdfstring{$\deg$}{deg}}\label{Subsubsection the second change}

\begin{Def}
Given a cycle $\Gamma$ defined over $K$ as the sum $\Gamma=\sum_i m_i[V_i]$, where each $V_i$ is irreducible over $K$, we define
\[
    \operatorname{t}(\Gamma)
        :=
    \sum_i m_i\operatorname{t}(V_i)
.\]
\end{Def}

The second change is to replace $\deg$ with our notion of $\operatorname{t}$.
For this, we use $P\in\mathfrak p-\{0\}$ with minimal $\operatorname{t}(P)$. We define $t:=\operatorname{t}(P)$. Each time $d$ appears in the original proof, we use $t$ instead.

The only auxiliary results used in the original proof affected by using $\operatorname{t}$ instead of $\deg$ are \cite[Theorem 2]{Binyamini2014MultiplicityEstimatesAnalyticCyclesAndNewtonPolytopesV2} and \cite[Lemma 12]{Binyamini2014MultiplicityEstimatesAnalyticCyclesAndNewtonPolytopesV2}. Since $\deg\leq\operatorname{t}$, \cite[Theorem 2]{Binyamini2014MultiplicityEstimatesAnalyticCyclesAndNewtonPolytopesV2} also applies with $\operatorname{t}$ instead of $\deg$. We replace \cite[Lemma 12]{Binyamini2014MultiplicityEstimatesAnalyticCyclesAndNewtonPolytopesV2} with \iref{Corollary existence of polynomial of specified t in ideal with small t}, which gives a constant different from $A_n$, denoted by $A'$.

In addition to replacing $A_n$ with $A'$, we need to adjust the constants in the following calculations from the original proof:
\begin{itemize}
    \item In the original proof, the author finds a bound for the degree of $Q_i^j\not\equiv 0$, where $Q_i^j=\xi^a(P)$ for some $1\leq a\leq A_n 2^n n$. For this, they use the inequality $\deg\xi^a(P)\leq\deg P+a\cdot\delta$, where $\delta$ is the maximal degree of the polynomials defining $\xi$. The $2$ here comes from the bound $\deg\Gamma^j\leq 2^j d^j$.
    
    In our case, since we are working with $\operatorname{t}$ instead of $\deg$, we use \iref{Corollary t(xi^a(f))<t(f)+a*ln(t(f)+a)} to establish that $\operatorname{t}(Q_i^j)\leq\operatorname{t}(P)+B\cdot a\cdot\ln(\deg P+B\cdot a)$ for some constant $B$ depending only on $\xi,K$ and for some $1\leq a\leq A' E^n n$, where $E$ is some constant that we will establish later and that depends only on $\xi,K$. We also need to ensure that $\operatorname{t}(\Gamma^j)\leq E^j t^j$.
    \item Then, using Bézout's theorem, the author obtains $\deg([W_i^j]\cdot\operatorname{V}(Q_i^j))\leq\deg W_i^j\cdot\deg Q_i^j\leq\deg W_i^j\cdot(d+a\cdot\delta)$. This, together with the assumption that $d\geq A_n 2^n n \delta$, gives the bound $\deg\Gamma^{j+1}\leq (2d)\cdot\deg\Gamma^j$.

    In our case, we use \iref{Remark Product of cycle over K with function} to obtain the bound $\operatorname{t}([W_i^j]\cdot \operatorname{V}(Q_i^j))\leq D\cdot\operatorname{t}(W_i^j)\operatorname{t}(Q_i^j)\leq\operatorname{t}(W_i^j)\cdot D\cdot(t+BA'E^n n\cdot\ln(t+B\cdot A'E^n n))$, where $D$ is some constant depending only on $n,K$. Thus, assuming that $t$ is large enough so that $t\geq BA'E^n n\cdot\ln(t+BA'E^n n)$, we obtain the bound $\operatorname{t}(\Gamma^{j+1})\leq(2Dt)\cdot\operatorname{t}(\Gamma^j)$. Thus, we may choose $E$ to be the constant $2D$.
\end{itemize}

Thus, the modified proof yields $\operatorname{mult}_V^\xi P\leq\mathcal N(n,\delta,F)$, where $F$ is the smallest integer for which $F\geq BA'(2D)^n n\cdot\ln(F+BA'(2D)^n n)$, completing the proof.

\section{The main theorem}\label{Section main theorem}

\subsection{The arithmetic D-property}\label{Subsection the arithmetic D-property}

To establish our results, we assume a non-degeneracy condition on our vector field $\xi$ and trajectory $\Gamma$. We call this condition the 'arithmetic D-property', since it is similar to the classical 'D-property' (see, for example, \cite[Definition on page 1346]{Nesterenko1996ModularFunctionsAndTranscendenceQuestions}). In this subsection, we define the arithmetic D-property and show that it holds in a common setting (see \iref{Subsubsection finite description of invariant varieties implies the arithmetic D-property}).

\subsubsection{Definition of the arithmetic D-property}

\begin{Def}\label{Definition arithmetic D-property}
A homogeneous vector field $\xi$ defined on $\mathbb C^{n+1}-\{0\}$, together with a trajectory $\Gamma$, is said to satisfy the arithmetic D-property over $K$ if there exists a constant depending only on $\xi,\Gamma,K$, denoted by $c^\mathrm{D}$, such that:
Every homogeneous ideal $\mathfrak p\subseteq K[x_0,...,x_n]$ satisfying
\begin{enumerate}
    \item  $\mathfrak p$ is prime.
    \item $\dim \mathfrak p>0$.
    \item  $\mathfrak p\supseteq\{Q,\xi(Q),\xi(\xi(Q)),...\}$ for some nonzero homogeneous $Q\in K[x_0,...,x_n]$.
\end{enumerate}
must also satisfy:
\[
    \operatorname*{ord}_{\Gamma}\mathfrak p
        \leq
    c^\mathrm{D}\cdot\operatorname{t}(\mathfrak p)^{\frac{n}{\operatorname{codim}\mathfrak p}}
.\]
\end{Def}

\begin{Rem}
If $(\xi,\Gamma)$ satisfies the arithmetic D-property over $K$, then $\Gamma$ cannot be contained in the zero locus of a nonzero homogeneous polynomial $Q\in K[x_0,...,x_n]$ (otherwise, the prime component of $\sqrt{\{Q,\xi(Q),\xi(\xi(Q)),...\}}$ whose zero locus contains $\Gamma$ would constitute a contradiction to the arithmetic D-property over $K$, since its order would be infinity).
\end{Rem}

\subsubsection{Finite description of invariant varieties implies the arithmetic D-property}\label{Subsubsection finite description of invariant varieties implies the arithmetic D-property}

In this subsubsection, we describe useful settings in which the arithmetic D-property holds.

\begin{Prop}\label{Proposition d property when distance is large}
  Suppose there exists a constant $A$ depending only on $\xi,\Gamma,K$ such that for any homogeneous prime $\mathfrak p\subseteq K[x_0,...,x_n]$ which contains some nonzero $\{Q,\xi(Q),\xi(\xi(Q)),...\}$, we have
  \begin{equation*}
    -\ln\max_{\omega\in\Gamma}\rho_{\omega,\mathfrak p}\leq A \cdot \operatorname{t}(\mathfrak p)^{\frac n {\operatorname{codim}\mathfrak p}-1}
  .\end{equation*}
Then $(\xi,\Gamma)$ satisfies the arithmetic D-property over $K$.
\end{Prop}

\begin{proof}
By \iref{Lemma Nesterenko zero of ideal close to omega},
\[
\forall\omega\in\mathbb C^{n+1}-\{0\}:\quad
    \deg\mathfrak p\cdot\ln\rho_{\omega,\mathfrak p}
        \leq
    \frac1{1+\dim\mathfrak p}\left(\ln|\mathfrak p(\omega)|+\operatorname{h}(\mathfrak p)\right)+\mathrm{const}\cdot\deg\mathfrak p
,\]
which, using our assumption, implies that
\begin{align*}
    \operatorname*{ord}_{\Gamma}\mathfrak p
        =
    \min_{\omega\in\Gamma}-\ln|\mathfrak p(\omega)|
        &\leq
        \min_{\omega\in\Gamma}-\ln\operatorname{\rho}_{\omega,\mathfrak p}\cdot(1+\dim\mathfrak p)\cdot\deg\mathfrak p+\mathrm{const}\cdot\operatorname{t}(\mathfrak p)
        \\&\leq
    \mathrm{const}\cdot\operatorname{t}(\mathfrak p)^\frac{n}{\operatorname{codim}\mathfrak p}
.\end{align*}
\end{proof}

The following corollary covers a common setting for the classical D-property:

\begin{Cor}
Assume that $\xi$ satisfies the following:
There exists a finite collection of projective varieties in $\mathbb C^{n+1}-\{0\}$ such that any projective irreducible $\xi$-invariant variety (i.e. a variety which is tangential to $\xi$ at every point) defined over $K$ (i.e. which is the zero locus in $\mathbb C^{n+1}-\{0\}$ of some homogeneous ideal $\subseteq K[x_0,...,x_n]$) must be contained in one of these varieties.
Then $(\xi,\Gamma)$ satisfies the arithmetic D-property over $K$ for any $\Gamma$ that is not contained in any of the varieties in the collection.
\end{Cor}

\begin{Rem}
Equivalently, we can state this corollary as follows:
Assume that $\xi$ satisfies the following:
There exists a finite collection of nonzero homogeneous polynomials in $\mathbb C[x_0,...,x_n]$ such that any $\sqrt{\{Q,\xi(Q),\xi(\xi(Q)),...\}}$ for any nonzero homogeneous  $Q\in K[x_0,...,x_n]$ must contain one of these polynomials.
Then $(\xi,\Gamma)$ satisfies the arithmetic D-property over $K$ for any $\Gamma$ that is not contained in the zero locus of any of the polynomials in the collection.
\end{Rem}

\begin{proof}
For each variety $W$ in this collection, we fix some point $\omega_W\in\Gamma-W$. Assume we have some homogeneous prime $\mathfrak p\subseteq K[x_0,...,x_n]$ which contains some $\{Q,\xi(Q),\xi(\xi(Q)),...\}$ for some nonzero homogeneous $Q\in K[x_0,...,x_n]$. Then, $\operatorname{V}(\mathfrak p)$ must be contained in the zero locus of $\{Q,\xi(Q),\xi(\xi(Q)),...\}$, which is a $\xi$-invariant variety, and thus, $\operatorname{V}(\mathfrak p)$ must be contained in some $W_0$ of the collection. We have:
\[
\max_{\omega\in\Gamma}\rho_{\omega,\operatorname{V}(\mathfrak p)}\geq\rho_{\omega_{W_0},\operatorname{V}(\mathfrak p)}\geq\rho_{\omega_{W_0},W_0}
.\]
Thus, $-\ln\max_{\omega\in\Gamma}\rho_{\omega,\mathfrak p}$ is bounded above by a constant $A$, and we obtain:
  \begin{equation*}
    -\ln\max_{\omega\in\Gamma}\rho_{\omega,\mathfrak p}\leq A \cdot \operatorname{t}(\mathfrak p)^{\frac n {\operatorname{codim}\mathfrak p}-1}
  ,\end{equation*}
which completes the proof.
\end{proof}

\subsection{Statement of the main theorem}

\begin{ThmCor}[The main theorem]\label{Theorem lower bound main theorem}
Assume that our fixed $(\xi,\Gamma)$ satisfies the arithmetic D-property over $K$. Then, for any homogeneous unmixed proper nontrivial ideal $I\subseteq K[x_0,...,x_n]$, we have:
\[
    \operatorname*{ord}_{\Gamma}I
        \leq
    C\cdot\operatorname{t}(I)^\frac{n}{\operatorname{codim}I}
,\]
where $C$ is some constant depending only on $\xi,\gamma,K$.
\end{ThmCor}

\subsection{Corollaries}

Assuming the main theorem, we establish the following corollaries:

\begin{Cor}\label{Corollary lower bound on I}
With $\xi,\gamma,K,I$ as before, there exists a constant $C$ depending only on $\xi,\gamma,K$ so that for any $0<h<\frac1e$, there are finitely many discs $\{D^i\}_{i=1}^a$ whose diameters sum to less than $h$, and with $a\leq C\cdot\operatorname{t}(I)^\frac{n}{\operatorname{codim}I}$, such that:
\[
    \forall\omega\in\gamma(\overline{D_\frac1e}-\bigcup_i D^i):\quad
    \ln|I(\omega)|
        \geq
    -C\cdot\ln\frac1h\cdot\operatorname{t}(I)^\frac{n}{\operatorname{codim}I}
.\]
\end{Cor}

\begin{proof}
By \iref{Theorem lower bound main theorem}:
\[
    \operatorname*{ord}_{\Gamma}I
        \leq
    \mathrm{const}\cdot\operatorname{t}(I)^\frac{n}{\operatorname{codim}I}
.\]
Thus, by \iref{Lemma lower bound on I in general}, there are finitely many discs $\{D^i\}_{i=1}^a$ whose diameters sum to less than $h$, such that:
\[
    \forall x\in\overline{D_\frac1e}-\bigcup_i D^i:\quad
    \ln|I(\gamma(x))|
        \geq
    -\mathrm{const}\cdot\ln\frac1h\cdot\operatorname{t}(I)^\frac{n}{\operatorname{codim}I}
\]
and
\[
    a
        \leq
    \mathrm{const}\cdot\operatorname{t}(I)^\frac{n}{\operatorname{codim}I}
.\]
\end{proof}

\begin{Cor}[see \iref{Corollary introduction; main corollary}]\label{Corollary the lower bound main theorem; main corollary}
With $\xi,\gamma,K$ as before, there exists a constant $C$ depending only on $\xi,\gamma,K$ such that for any homogeneous $p\in K[x_0,...,x_n]$ and any $0<h<\frac1e$, there is a collection of discs $\{D^i\}_{i=1}^a$ whose diameters sum to less than $h$, and with $a\leq C\cdot\operatorname{t}(p)^n$, such that:
\[
    \forall\omega\in\gamma(\overline{D_\frac1e}-\bigcup_i D^i):\quad
    \ln|p(\omega)|
        \geq
    \ln|p|-C\cdot\ln\frac1h\cdot\operatorname{t}(p)^n
.\]
\end{Cor}

\begin{proof}
Let $I\subseteq K[x_0,...,x_n]$ be the ideal generated by $p$. 
By \iref{Lemma Nesterenko principal ideal},
\begin{align*}
    \forall\omega\in\mathbb C^{n+1}-\{0\}&:\\
        \ln|I(\omega)|
            &\leq
        \ln|p(\omega)|+\ln\left(|p|^{-1}\right)+\ln\left(|\omega|^{-\deg p}\right)+\mathrm{const}\cdot\deg p
\end{align*}
and
\[
    \operatorname{t}(I)
        \leq
    \mathrm{const}\cdot\operatorname{t}(p)
.\]
So, by \iref{Corollary lower bound on I}, there are finitely many discs $\{D^i\}_{i=1}^a$ whose diameters sum to less than $h$, and with $a\leq\mathrm{const}\cdot\operatorname{t}(I)^n$, such that (recall that $\operatorname{codim}I=1$):
\begin{align*}
    \forall\omega\in\gamma(\overline{D_\frac1e}-\bigcup_i D^i):\quad
    \ln|I(\omega)|
        &\geq
    -\mathrm{const}\cdot\ln\frac1h\cdot\operatorname{t}(I)^n\\
        &\geq
    -\mathrm{const}\cdot\ln\frac1h\cdot\operatorname{t}(p)^n
\end{align*}
Since $|\omega|=1$ for any $\omega\in\Gamma$, we obtain:
\begin{align*}
    \forall\omega\in\gamma(\overline{D_\frac1e}-\bigcup_i D^i)&:\\
    \ln|p(\omega)|&+\ln\left(|p|^{-1}\right)+\mathrm{const}\cdot\deg p
        \geq
    -\mathrm{const}\cdot\ln\frac1h\cdot\operatorname{t}(p)^n
.\end{align*}
So,
\begin{align*}
    \forall\omega\in\gamma(\overline{D_\frac1e}-\bigcup_i D^i):&\\
    \ln|p(\omega)|
        &\geq
    \ln|p|
    -\mathrm{const}\cdot\ln\frac1h\cdot\operatorname{t}(p)^n\\
\end{align*}
\end{proof}

\begin{Cor}[see \iref{Theorem introduction; main theorem}]\label{Corollary the lower bound main theorem; main remark}
With $\gamma,p$ as before:
\[
    \max_{\omega\in\Gamma}\ln|p(\omega)|
        \geq
    \ln|p|-C\cdot\operatorname{t}(p)^n
,\]
for some constant $C$ depending only on $\xi,\gamma,K$.
\end{Cor}

\begin{proof}
This follows from \iref{Corollary the lower bound main theorem; main corollary} by choosing $h=0.01$.
\end{proof}

\begin{Cor}\label{Corollary number of zeroes with multiplicity}
With $\gamma,p$ as before, we have:
\[
    \B_{\frac1e,1}(p\circ\gamma)
        \leq
    C\cdot\operatorname{t}(p)^n
,\]
for some constant $C$ depending only on $\xi,\gamma,K$.
Thus, the number of zeros of $p\circ\gamma$ in $\overline{D_\frac1e}$ (counted with multiplicities) is at most
\[
    \#\mathrm{Zeroes}_\frac1e(p\circ\gamma)\leq C'\cdot\operatorname{t}(p)^n
\]
for some constant $C'$ depending only on $\xi,\gamma,K$.
\end{Cor}

\begin{proof}
The first inequality follows from \iref{Corollary the lower bound main theorem; main remark} together with the bound $\max_{\omega\in\Gamma_1}|p(\omega)|\leq|p|\cdot\mathrm{const}^{\deg p}$. The bound on the number of zeros follows from \iref{Lemma bound on number of zeroes of f given bernstein index}.
\end{proof}

\section{Proof of the main theorem}\label{Section lower bound main theorem}

\subsection{Outline of the induction}

We prove \iref{Theorem lower bound main theorem} by induction on $m$, the dimension of the ideal. In each of the following cases, we establish the theorem with a constant depending only on $\xi,\gamma,K$ and the case under consideration:
\begin{enumerate}
\item \emph{Base case 1:} we prove the claim for homogeneous prime ideals of dimension $0$, establishing a constant $c_0^\mathrm{prime}$.
\item \emph{Base case 2:} we prove the claim for homogeneous prime ideals of any dimension larger than $0$ that contain some $\{Q,\xi(Q),\xi(\xi(Q)),...\}$ for some nonzero homogeneous $Q\in K[x_0,...,x_n]$, establishing a constant $c^\mathrm{invariant}$.
\item \emph{Intermediate reduction in dimension $m$:} assuming the claim holds for all homogeneous prime ideals of dimension $m$ (for some $0\leq m\leq n-1$) with some constant $c_m^\mathrm{prime}$, we prove it for all homogeneous unmixed ideals of dimension $m$, establishing a constant $c_m$.
\item \emph{Induction step from dimension $m-1$ to dimension $m$:} assuming the claim holds for all homogeneous unmixed ideals of dimension $m-1$ (for some $0\leq m-1\leq n-2$) with some constant $c_{m-1}$, we prove it for all homogeneous prime ideals of dimension $m$, establishing a constant $c_m^\mathrm{prime}$.
\end{enumerate}

Together, these cases give, for any $I$:
\[
    \operatorname*{ord}_{\Gamma}I
        \leq
    \max(c_0,...,c_{n-1})\cdot\operatorname{t}(I)^\frac{n}{\operatorname{codim}I}
,\]
which establishes the theorem.

\subsection{Base case: prime \texorpdfstring{$\mathfrak p$}{p} with \texorpdfstring{$\dim\mathfrak p=0$}{dim p=0}}

\subsubsection{Main ingredient}

For this case, we will use the following lemma:

\begin{Lem}\label{Lemma aux E ln(x-omega)>const}
Let $\tau:\overline{D_{a,r}}\rightarrow\{1\}\times\overline{D_\frac12}^n$ be of the form
\[
    \tau(z):=(1,z,\tau'(z))
,\]
for some $\tau':\overline{D_{a,r}}\to\overline{D_\frac12}^{n-1}$, where $\overline{D_{a,r}}$ is the disc in the complex plane centered at $a\in\mathbb C$ with radius $r\in\mathbb R_{>0}$, and $|a|+r\leq\frac12$. Then:
\[
    \forall\omega\in\mathbb C^{n+1}-\{0\}:\quad
    \int_{z\in D_{a,r}}\min(0,\ln\|\tau(z)-\omega\|)dA
        \geq
    -\frac\pi2
,\]
where $dA$ is the area measure.
\end{Lem}

\begin{proof}
Write $\omega=:(\omega_0,...,\omega_n)$. If $\max_i|w_i|$ is attained at some $i_\omega\neq 0$, then, since $|\tau(z)|=1$ and $|\tau(z)_{i_\omega}|\leq\frac12$, we have:
\begin{align*}
    \forall z \in \overline{D_{a, r}}:\quad
    \|\tau(z) - \omega\| 
    &\geq \left|1 \cdot \omega_{i_\omega} - \tau(z)_{i_\omega} \cdot \omega_0\right| \cdot |\omega_{i_\omega}|^{-1}
    \\&\geq \left(|\omega_{i_\omega}| - \left|\frac{\omega_0}{2}\right|\right) \cdot |\omega_{i_\omega}|^{-1} 
    \geq \frac12
,\end{align*}
and since $r\leq\frac12$, the integral is at least $-0.55$.

Otherwise, $\max_i|\omega_i|$ is attained at $\omega_0$. By renormalizing $\omega$, we may assume that $\omega_0=1$ and $|\omega|=1$. Then, by \iref{Remark projective vs affine distance}, we have:
\[
    \forall z\in D_{a,r}:\quad
        \|\tau(z)-\omega\|
            =
        \|\tau(z)-\omega\|_\infty
            \geq
        |z-\omega_1|
    .\]
Now we can calculate:
\begin{align*}
    \int_{z\in D_{a, r}} \min(0,\ln \|\tau(z) - \omega\|)dA
    &\geq \int_{z\in D_{a, r}}\min(0,\ln |z - \omega_1|)dA\\
    &=\int_{x\in D_{a-\omega_1, r}} \min(0,\ln|x|)dA
,\end{align*}
and by considering only the negative values of $\ln|x|$, we see that this integral is at least:
\[
    \int_{x\in\overline{D_1}}\ln|x|dA=-\frac\pi2
.\]
This completes the proof.
\end{proof}

\subsubsection{Establishing this case}

By \iref{Lemma zero dimensional ln|p(omega)|}, there exist $\omega_1,...,\omega_{\deg\mathfrak p}\in\mathbb C^{n+1}-\{0\}$ such that
\[
    \forall\omega\in\mathbb C^{n+1}-\{0\}:\quad
    \ln|\mathfrak p(\omega)|
        \geq
    \sum_i\ln\|\omega-\omega_i\|
        -
    \mathrm{const}\cdot\deg\mathfrak p
.\]

Since $\frac{d\gamma}{dz}\not\equiv 0$, we may assume without loss of generality that there exists a holomorphic trajectory
\begin{align*}
\begin{cases}
    \tau' &: D_{a,r}\to\mathbb C^{n-1}\\
    \tau&: D_{a,r}\to\mathbb C^{n+1}-\{0\}\\
    \tau(z)&:=(1,z,\tau(z)')
\end{cases}
,\end{align*}
for some $a$ and $0<r\leq 1$ ($\tau',\tau,a,r$ depend only on $\Gamma$), satisfying, for any $z$, that $\tau(z)\in\Gamma$.

By \iref{Lemma aux E ln(x-omega)>const}, since $r$ is a constant depending only on $\xi$, for each $i$,
\[
    \operatorname{\mathbb E}_{z\in D_{a,r}}(\ln\|\tau(z)-\omega_i\|)
        \geq
    \operatorname{\mathbb E}_{z\in D_{a,r}}(\min(0,\ln\|\tau(z)-\omega_i\|))
        \geq
-\mathrm{const}
,\]
where $\mathbb E$ is the expectation operator. Thus, by linearity of expectation, we have
\[
    \operatorname{\mathbb E}_{z\in D_{a,r}}
    \ln|\mathfrak p(\tau(z))|
        \geq\\
    -\mathrm{const}\cdot\deg\mathfrak p
,\]
hence there exists some $z_0\in D_{a,r}$ for which, writing $x:=\tau(z_0)\in\Gamma$, we have:
\[
    \ln|\mathfrak p(x)|
        \geq\\
    -\mathrm{const}\cdot\deg\mathfrak p
.\]
Consequently,
\[
    \operatorname*{ord}_{\Gamma}\mathfrak p
        \leq
    \mathrm{const}\cdot\deg\mathfrak p
        \leq
    \mathrm{const}\cdot\operatorname{t}(\mathfrak p)
,\]
and we may choose $c_0^\mathrm{prime}$ to be this constant.

\subsection{Base case: the arithmetic D-property}

Let $\mathfrak p$ be a prime ideal with $\dim\mathfrak p>0$ containing some nonzero $\{Q,\xi(Q),\xi(\xi(Q)),...\}$.
Since $(\xi,\Gamma)$ satisfies the arithmetic D-property over $K$,
\[
    \operatorname*{ord}_{\Gamma}\mathfrak p
        \leq
    c^\mathrm{D}\cdot\operatorname{t}(\mathfrak p)^{\frac{n}{\operatorname{codim}\mathfrak p}}
,\]
and thus we take $c^\mathrm{invariant}$ to be $c^D$.

\subsection{Intermediate reduction: reducing from the unmixed case to the prime case}

\subsubsection{Main ingredient}

The main ingredient for this case is the following proposition:

\begin{Prop}\label{Proposition I = intersection of pi}
Let $I\subseteq K[x_0,...,x_n]$ be a homogeneous unmixed ideal. Consider its reduced primary decomposition, with prime radicals $\mathfrak p_i$ of multiplicity $n_i$. Then:
\[
    \operatorname*{ord}_{\Gamma}I
        \leq
    C_1\cdot\sum_i n_i\operatorname*{ord}_{\Gamma}\mathfrak p_i
        +
    C_2\cdot\deg I
,\]
where $C_1,C_2$ are some constants depending only on $\xi,\gamma,K$.
\end{Prop}

\begin{proof}
By \iref{Lemma Nesterenko primary decomposition of ideal}, we have
\[
\forall\omega\in\mathbb C^{n+1}-\{0\}:
\sum_i n_i \ln|\mathfrak p_i(\omega)|
    \leq
\mathrm{const}\cdot\deg I+\ln|I(\omega)|
.\]
Thus,
\[
    \max_{\omega\in\Gamma}\ln|I(\omega)|
        \geq
    \max_{\omega\in\Gamma}\left(\sum_i n_i \ln|\mathfrak p_i(\omega)|\right)
        -
    \mathrm{const}\cdot\deg I
.\]

Any measurable function $f:X\rightarrow[-\infty,\infty)$ satisfies $\max_{x\in X} f\geq\operatorname{\mathbb E}_{x\in X}(f)$, so, regarding $\overline{D_\frac1e}$ as a probability space with the uniform measure, we obtain:
\begingroup\small
\[
    \max_{\omega\in\Gamma}\left(\sum_i n_i \ln|\mathfrak p_i(\omega)|\right)
        \geq
    \operatorname{\mathbb E}_{z\in D_\frac1e}\left(\sum_i n_i \ln|\mathfrak p_i(\gamma(z))|\right)
        =
    \sum_i n_i\operatorname{\mathbb E}_{z\in D_\frac1e}\left(\ln|\mathfrak p_i(\gamma(z))|\right)
.\]
\endgroup

Since each $|\mathfrak p_i|\circ\gamma$ is defined as a maximum of holomorphic functions over $\overline{D_1}$ (because $|\gamma(z)|=1$), we can apply \iref{Lemma average is like maximum minus Bernstein index} to obtain:
\begin{align*}
    \operatorname{\mathbb E}_{z\in\overline{D_\frac1e}}\ln|\mathfrak p_i(\gamma(z))|
        &\geq
    \ln\max_{z\in\overline{D_\frac1e}}|\mathfrak p_i(\gamma(z))|-\mathrm{const}\cdot\B_{\frac1e,1}(|\mathfrak p_i|\circ\gamma)
        \\&=
    -\operatorname*{ord}_{\Gamma}\mathfrak p_i-\mathrm{const}\cdot\B_{\frac1e,1}(|\mathfrak p_i|\circ\gamma)
.\end{align*}
By \iref{Lemma Bernstein index of I},
\[
    \B_{\frac1e,1}(|\mathfrak p_i|\circ\gamma)
        \leq
    \operatorname*{ord}_{\Gamma}\mathfrak p_i
    +
    \mathrm{const}\cdot\deg \mathfrak p_i
.\]
Combining these equations gives:
\begin{align*}
    \max_{\omega \in \Gamma} \ln |I(\omega)|
    &\geq \sum_i n_i \left(
    -\operatorname*{ord}_{\Gamma} \mathfrak{p}_i 
    -\mathrm{const}\cdot\operatorname*{ord}_{\Gamma} \mathfrak{p}_i 
    -\mathrm{const} \cdot \deg \mathfrak{p}_i
    \right) - \mathrm{const} \cdot \deg I \\
    &= -\mathrm{const} \sum_i n_i \operatorname*{ord}_{\Gamma} \mathfrak{p}_i 
    - \mathrm{const} \sum_i n_i \deg \mathfrak{p}_i 
    - \mathrm{const} \cdot \deg I
.\end{align*}

Now,
\[
    \sum_i n_i \deg\mathfrak p_i=\deg I
.\]
and so,
\[
    \max_{\omega\in\Gamma}\ln|I(\omega)|
        \geq
    -\mathrm{const}\sum_i n_i\operatorname*{ord}_{\Gamma}\mathfrak p_i
    -
    \mathrm{const}\cdot\deg I
.\]
That is,
\[
\operatorname*{ord}_{\Gamma}I
    \leq
 \mathrm{const}\cdot\sum_i n_i\operatorname*{ord}_{\Gamma}\mathfrak p_i
    +
\mathrm{const}\cdot\deg I
 .\]
 \end{proof}

\subsubsection{Establishing this case}

We are given a homogeneous unmixed ideal $I\subseteq K[x_0,...,x_n]$. We consider its reduced primary decomposition with prime radicals $\mathfrak p_i$ of multiplicities $n_i$.
Recall that in this case we assume the main theorem holds for all \textbf{prime} ideals of dimension $m:=\dim I$, with constant $c_m^\mathrm{prime}$.

By \iref{Proposition I = intersection of pi} and the induction hypothesis for the prime case applied to each $\mathfrak p_i$, we have
\[
    \operatorname*{ord}_{\Gamma}I
        \leq
    \mathrm{const}\cdot c_m^\mathrm{prime}\cdot\sum_i n_i\cdot\operatorname{t}(\mathfrak p_i)^\frac{n}{\operatorname{codim}I}
        +
\mathrm{const}\cdot\deg I
,\]
and by \iref{Lemma Nesterenko primary decomposition of ideal}, we have
\[
    \sum_i n_i\operatorname{t}(\mathfrak p_i)
        \leq
    \mathrm{const}\cdot\operatorname{t}(I)
.\]
Using these equations and the fact that $x^a+y^a\leq (x+y)^a$ for $x,y\geq 0,\,a\geq 1$, we obtain:
\begin{align*}
    \operatorname*{ord}_{\Gamma} I
    &\leq \mathrm{const} \cdot c_m^\mathrm{prime}\cdot\left(\sum_i n_i \operatorname{t}(\mathfrak{p}_i)\right)^{\frac{n}{\operatorname{codim} I}}
    + \mathrm{const} \cdot \deg I \\
    &\leq \mathrm{const} \cdot(c_m^\mathrm{prime} + 1)\cdot\operatorname{t}(I)^{\frac{n}{\operatorname{codim} I}}
,\end{align*}
and thus we take $c_m$ to be $\mathrm{const}\cdot(c_m^\mathrm{prime}+1)$.

\subsection{The induction step: prime \texorpdfstring{$\mathfrak p$}{p} with \texorpdfstring{$\dim\mathfrak p>0$}{dim p>0}}

\subsubsection{Main ingredient}

For this case, we use the following proposition:

\begin{Prop}\label{Proposition construction of G}
 Let $\mathfrak p\subseteq K[x_0,...,x_n]$ be a homogeneous prime ideal that does not contain any nonzero $\{Q,\xi(Q),\xi(\xi(Q)),...\}$. Then there exists a homogeneous polynomial $G\in K[x_0,...,x_n]$ that satisfies the following properties:
\begin{enumerate}
    \item $G\not\in\mathfrak p$.
    \item $\operatorname{t}(G)\leq C\cdot\operatorname{t}(\mathfrak p)^\frac1{\operatorname{codim}\mathfrak p}$.
    \item  $\|G\|_{\gamma(z)}
        \leq
    e^{C\cdot\operatorname{t}(\mathfrak{p})^{\frac1{\operatorname{codim} \mathfrak{p}}}}\cdot(\rho_{\gamma(z), \mathfrak{p}} + \rho_{\gamma(z), \mathfrak{p}}^{0.6})$ for all $z$ in some large measurable subset $A$ of $D_\frac1e$ (more precisely, $\operatorname{Area}(A)\geq 0.56\cdot\operatorname{Area}(\overline{D_\frac1e})$, where $\operatorname{Area}$ is used in the usual Euclidean sense),
\end{enumerate}
where $C$ is some constant depending only on $\xi,\gamma,K$.
\end{Prop}

\begin{proof}
By \iref{Proposition existence of polynomial of small t and small mult in ideal}, $\exists P\in\mathfrak p$ such that:
\[
    \operatorname{t}( P)
        \leq
    \mathrm{const}\cdot\operatorname{t}(\mathfrak p)^\frac1{\operatorname{codim}\mathfrak p}
\]
and
\[
    \operatorname{mult}_{\mathfrak p}^{\xi} P
        \leq
    \mathrm{const}
\]
($\operatorname{mult}_{\mathfrak p}^{\xi} P$ equals the number of times we need to differentiate $P$ along the vector field $\xi$ for it to leave $\mathfrak p$).
\begin{itemize}
    \item We define $F$ to be the derivative of order $\operatorname{mult}_{\mathfrak p}^{\xi} P-1$ of $ P$ along $\xi$ (since $ P\in\mathfrak p$, we have $\operatorname{mult}_{\mathfrak p}^{\xi} P\geq 1$).
    \item We define $G$ to be $\xi(F)$. Since $G$ is the derivative of order $\operatorname{mult}_{\mathfrak p}^{\xi} P$ of $ P$ along the vector field $\xi$, we have $G\not\in\mathfrak p$.
\end{itemize}

We note that $F\not\equiv 0$ since $G\not\in\mathfrak p$. It follows from \iref{Corollary t(xi^a(f))<t(f)+a*ln(t(f)+a)} that $\operatorname{t}(F)
            \leq
            \mathrm{const}\cdot\operatorname{t}(\mathfrak p)^\frac1{\operatorname{codim}\mathfrak p}
        $ and $\operatorname{t}(G)
            \leq
            \mathrm{const}\cdot\operatorname{t}(\mathfrak p)^\frac1{\operatorname{codim}\mathfrak p}
        $ and that for any $x$ which is a zero of $\mathfrak p$, we have $F(x)=0$. Thus, by \iref{Lemma bound on xi(f)(omega)} and \iref{Lemma f(omega)<|f|*rho} we have
\begin{align*}
    \forall \omega \in \gamma(A),x\in \operatorname{V}(\mathfrak p): \quad
    \|G\|_\omega
    &\leq e^{\mathrm{const} \cdot \operatorname{t}(F)}\cdot (\|F\|_\omega + \|F\|_\omega^{0.6})\\
    &\leq e^{\mathrm{const} \cdot\operatorname{t}(\mathfrak p)^{\frac1{\operatorname{codim}\mathfrak p}}}\cdot(\|\omega-x\|+\|\omega-x\|^{0.6})
,\end{align*}
and the conclusion follows.
\end{proof}

\subsubsection{Establishing this case}

This is the main case.
For this case, we make two assumptions:
\begin{enumerate}
\item We assume that $\mathfrak p$ does not contain any nonzero $\{Q,\xi(Q),\xi(\xi(Q)),...\}$. Otherwise, one of the base cases gives $
    \operatorname*{ord}_{\Gamma}\mathfrak p
        \leq
    c^\mathrm{invariant}\operatorname{t}(\mathfrak p)^{\frac{n}{\operatorname{codim}\mathfrak p}}
$, so the theorem holds provided we choose $c_m^\mathrm{prime}$ to be at least $c^\mathrm{invariant}$.
    \item We assume that $
    \operatorname*{ord}_{\Gamma}\mathfrak p
        >
    \alpha\cdot\operatorname{t}(\mathfrak p)^{1+\frac1{\operatorname{codim}\mathfrak p}}
$ for some constant $\alpha$ to be determined later. Otherwise, since $1+\frac1{\operatorname{codim}\mathfrak p}\leq\frac{n}{\operatorname{codim}\mathfrak p}$, the theorem holds provided we choose $c_m^\mathrm{prime}$ to be at least $\alpha$. We do not specify $\alpha$ explicitly. Instead, throughout this subsubsection we impose conditions requiring $\alpha$ to exceed constants that may depend on the fixed parameters, and finally choose any $\alpha$ satisfying all these conditions.

This assumption, together with \iref{Lemma upper bound on rho}, implies that
\begin{equation}\label{Equation upper bound on rho}
\forall\omega\in\Gamma:\,
        \rho_{\omega,\mathfrak p}
            <
        e^{-\frac{\alpha}{2n}\operatorname{t}(\mathfrak p)^{\frac1{\operatorname{codim}\mathfrak p}}}
\end{equation}
\end{enumerate}
for sufficiently large $\alpha$.

We now give the following construction:

\begin{Construction}[Construction of $G$]\label{Construction of G}
We construct a polynomial $G\in K[x_0,...,x_n]$ that satisfies the following properties:
\begin{enumerate}
    \item $G\not\in\mathfrak p$.
    \item \label{Property t(G) bound} $\operatorname{t}(G)\leq\mathrm{const}\cdot\operatorname{t}(\mathfrak p)^\frac1{\operatorname{codim}\mathfrak p}$.
    \item \label{Property G<rho^0.5} $\|G\|_{\gamma(z)}\leq\rho_{\gamma(z),\mathfrak p}^{0.5}$ for any $z$ in some large measurable subset $A$ of $D_\frac1e$ (more precisely, $\operatorname{Area}(A)\geq 0.56\cdot\operatorname{Area}(\overline{D_\frac1e})$, where $\operatorname{Area}$ is used in the usual Euclidean sense).
\end{enumerate}

$G$ is obtained from \iref{Proposition construction of G}. It follows that $G$ satisfies the first two properties and that:
\begin{equation}\label{Equation upper bound on |G(gamma(z))|}
    \|G_{\gamma(z)}\|
        \leq
    e^{\mathrm{const} \cdot \operatorname{t}(\mathfrak{p})^{\frac1{\operatorname{codim} \mathfrak{p}}}}\cdot (\rho_{\gamma(z), \mathfrak{p}} + \rho_{\gamma(z), \mathfrak{p}}^{0.6})
.\end{equation}
\iref{Property G<rho^0.5} follows from \iref{Equation upper bound on rho}: if $\alpha$ is at least $2n\ln\frac1{0.02}$ (so that $
    \rho_{\omega,\mathfrak p}+\rho_{\omega,\mathfrak p}^{0.6}
        \leq
    \rho_{\omega,\mathfrak p}^{0.55}
$) and $\alpha$ is at least $40n\cdot\mathrm{const}$ (for the $\mathrm{const}$ in \iref{Equation upper bound on |G(gamma(z))|}), then:
\[
    \forall\omega\in\gamma(A):\quad
    \|G_\omega\|
        \leq
    e^{\mathrm{const}\cdot\operatorname{t}(\mathfrak p)^\frac1{\operatorname{codim}\mathfrak p}}\cdot\rho_{\omega,\mathfrak p}^{0.55}
        \leq
    \rho_{\omega,\mathfrak p}^{0.5}
.\]
\end{Construction}

\begin{Construction}[Construction of $J$]\label{Construction of J}
Applying \iref{Corollary Nesterenko corollary intersection of ideal and polynomial} with $\omega:=\gamma(z),\,-S:=\ln|\mathfrak p(\gamma(z))|,\,\eta:=4, P:=G$ (for $z\in A$), we construct an ideal $J$ with the following properties:
    \begin{enumerate}
        \item $J$ is, in a cycle-theoretic sense, the intersection of $\mathfrak p$ and $(G^4)$. In particular, $\dim J=\dim\mathfrak p-1$.
            \item  $\operatorname{t}(J)\leq\mathrm{const}\cdot\operatorname{t}(\mathfrak p)\operatorname{t}(G)$.
        \item $    \forall z\in A:\,
        \ln|J(\gamma(z))|
          \leq
        \ln|\mathfrak p(\gamma(z))|
        +
        \mathrm{const}\cdot\operatorname{t}(\mathfrak p)\operatorname{t}(G)$.
    \end{enumerate}
    
(The second assumption gives $S>0$. By \iref{Property G<rho^0.5}, \iref{Property t(G) bound}, and \iref{Equation upper bound on rho}, if $\alpha$ is sufficiently large that $\frac{\alpha}{4n}\geq 2n[ K:\mathbb Q]\cdot\mathrm{const}$ (for the $\mathrm{const}$ in \iref{Property t(G) bound}), then
$
    \|G\|_{\gamma(z)}
        \leq
    e^{-2n[ K:\mathbb Q]\deg G}
$, verifying the conditions needed for this corollary.)
\end{Construction}

Finally, we explain why the construction of $G, J$ completes this case.
Applying \iref{Lemma lower bound on I in general} with $I:=J,\,h:=\frac1{2e}$, we obtain finitely many discs $\{D^i\}_{i=1}^a$ whose diameters sum to less than $\frac1e$ such that:
\[
    \forall x\in\overline{D_\frac1e}-\bigcup_i D^i:\quad
    \ln|J(\gamma(x))|
        \geq
    -\mathrm{const}\cdot\max(0,\operatorname*{ord}_{\Gamma}J)
        -
    \mathrm{const}\cdot\deg J
.\]

Since $\operatorname{Area}(A)\geq 0.56\cdot\operatorname{Area}\left(\overline{D_\frac1e}\right)$ and $\operatorname{Area}(\overline{D_\frac1e}-\bigcup_i D^i)\geq 0.75\cdot\operatorname{Area}\left(\overline{D_\frac1e}\right)$, the intersection of $A$ and $\overline{D_\frac1e}-\bigcup_i D^i$ is nonempty. For any point $z_0$ in this intersection:
\begingroup\small
\begin{align*}
    -\mathrm{const}\cdot\max(0,\operatorname*{ord}_{\Gamma}J) 
    - \mathrm{const} \cdot \deg J
    &\leq \ln |J(\gamma(z_0))|
    \leq \ln |\mathfrak{p}(\gamma(z_0))| 
    + \mathrm{const} \cdot \operatorname{t}(\mathfrak{p})^{1 + \frac1{\operatorname{codim} \mathfrak{p}}} \\
    &\leq \ln \max_{z \in \overline{D_{\frac1e}}} |\mathfrak{p}(\gamma(z))| 
    + \mathrm{const} \cdot \operatorname{t}(\mathfrak{p})^{1 + \frac1{\operatorname{codim} \mathfrak{p}}} \\
    &= -\operatorname*{ord}_{\Gamma} \mathfrak{p} 
    + \mathrm{const} \cdot \operatorname{t}(\mathfrak{p})^{1 + \frac1{\operatorname{codim} \mathfrak{p}}}
.\end{align*}
\endgroup
So,
\[
    \operatorname*{ord}_{\Gamma}\mathfrak p
        \leq
    \mathrm{const}\cdot\max(0,\operatorname*{ord}_{\Gamma}J)
    +
    \mathrm{const}\cdot\operatorname{t}(\mathfrak p)^{1+\frac1{\operatorname{codim}\mathfrak p}}
    +
    \mathrm{const}\cdot\deg J
.\]
By the induction hypothesis on $J$, the properties of $J$ and $G$ in \iref{Construction of J} and \iref{Construction of G}, and our assumption that $\dim\mathfrak p\geq 1$, the right-hand side is $\leq\mathrm{const}\cdot(c_{m-1}+1)\cdot\operatorname{t}(\mathfrak p)^\frac{n}{\operatorname{codim}\mathfrak p}$.
Thus, we may take $c_m^\mathrm{prime}:=\max(\mathrm{const}\cdot(c_{m-1}+1),c^\mathrm{invariant}, \alpha)$.

\section{Appendix: some calculations}\label{Section appendix calculations}

\subsection{Calculations for the norm of an ideal at a point}

\begin{Lem}\label{Lemma bound on ln|I(omega)|}
Let $I\subseteq K[x_0,...,x_n]$ be a non-trivial homogeneous unmixed ideal. Then
\[
    \forall\omega\in\mathbb C^{n+1}-\{0\}:\quad
        \ln|I(\omega)|
            \leq
        C\cdot\deg I
,\]
for some constant $C$ depending only on $n$.
\end{Lem}

\begin{proof}
Let $p_I(u_{1,0},...,u_{\dim I+1,n})$ be the Chow form of $I$.
Recall the following:
\begin{enumerate}
    \item The degree of $p_I$ is $(\dim I+1)\cdot\deg I\leq\mathrm{const}\cdot\deg I$.
    \item The definition of $\chi(p_I)(\omega)$:
\[
    \chi(p_I)(\omega)
        :=
    p_I(S^{(1)}\omega,...,S^{(\dim I+1)}\omega)
    \in
     K[s^{(i)}_{j,k} | j<k]
.\]
where $S^{(i)}:=
    \begin{pmatrix}
    0 & s^{(i)}_{0,1} & \cdots & s^{(i)}_{0,n} \\
    -s^{(i)}_{0,1} & 0 & \ddots & \vdots \\
    \vdots & \ddots & 0 & s^{(i)}_{n-1,n} \\
    -s^{(i)}_{0,n} & \cdots & -s^{(i)}_{n-1,n} & 0
    \end{pmatrix}
    $ is a skew-symmetric $(n+1)\times(n+1)$ matrix whose entries are variables, and $\omega$ is viewed as a column vector.
    \item The definition of $|I(\omega)|$:
\[
    |I(\omega)|
        :=
    |\chi(p_I)(\omega)|\cdot|p_I|^{-1}\cdot|\omega|^{-(\dim I + 1)\deg I}
.\]
\end{enumerate}

We now seek an upper bound for $|p_I(S^{(1)}\omega,...,S^{(\dim I+1)}\omega)|$:
\begin{itemize}
    \item $p_I$ is a polynomial in $(\dim I+1)(n+1)$ variables, so it has at most $((\dim I+1)(n+1))^{\deg p_I+1}$ terms.
    \item Substituting $\{S^{(i)}\omega\}_i$ into $p_I$ replaces each variable of $p_I$ with a bilinear expression of the form $\pm s_{i_1}\omega_{i_1}\pm...+\pm_{i_n}\omega_{i_n}$ (that is, $n$ summands).
    \item After substitution, each term becomes a product of $\deg p_I$ of these bilinear expressions.
    \item Expanding the brackets gives
\begin{align*}
    |p_I(S^{(1)}\omega,...,S^{(\dim I+1)}\omega)|
        &\leq
    ((\dim I+1)(n+1))^{\deg p_I+1}\cdot|p_I|\cdot n^{\deg p_I}\cdot|\omega|^{\deg p_I}\\
    &\leq
    \mathrm{const}^{\deg I}\cdot|p_I|\cdot|\omega|^{(\dim I+1)\deg I}
.\end{align*}
\end{itemize}
Consequently,
\[
    |I(\omega)|
        \leq
    \mathrm{const}^{\deg I}
.\]
\end{proof}

\begin{Lem}\label{Lemma Chow form of single point}
Let $\mathfrak p_x\subseteq\overline{\mathbb Q}[x_0,...,x_n]$ be a homogeneous prime ideal that has a single zero at some point $x=:
\begin{pmatrix}
    x^0 \\
    \vdots \\
    x^n
\end{pmatrix}
\in\mathbb C^{n+1}-\{0\}$. Then $\forall\omega\in\mathbb C^{n+1}-\{0\}$:
\[
    |\mathfrak p_x(\omega)|
        =
    \|\omega-x\|
.\]
\end{Lem}

\begin{proof}
Since $\mathfrak p_x$ has a unique zero, it is the ideal generated by all homogeneous polynomials that vanish at $x$. Thus, its Chow form is
\[
    F_{\mathfrak p_x}(u)
        =
    x^0u_0+...+x^nu_n
        =
    x^t\cdot u
.\]
Thus,
\[
    |\mathfrak p_x(\omega)|
        =
    |F_{\mathfrak p_x}(S\omega)|\cdot|F_{\mathfrak p_x}|^{-1}\cdot|\omega|^{-\deg F_{\mathfrak p_x}}
        =
    |x^tS\omega|\cdot|x|^{-1}\cdot|\omega|^{-1}
,\]
where $S$ is the skew-symmetric matrix in the variables $\{s_{i,j}|0\leq i<j\leq n\}$:
\[
    S
        :=
    \begin{pmatrix}
        0                    &   s_{0,1}       &    s_{0,2}    &   \cdots          &    s_{0,n}\\
        -s_{0,1}    &   0                     &   s_{1,2}     &   \cdots          &    s_{1,n}\\
        \vdots         & \ddots           &   \ddots       &   \ddots          &    \vdots\\
        -s_{0,n-1}& -s_{1,n-1} &   \ddots       &   0                      &    s_{n,n-1}\\
        -s_{0,n}    & -s_{1,n}      &   \cdots       &   -s_{n-1,n} &    0
    \end{pmatrix}
.\]

Now,
\[
    x^tS\omega
        =
    \sum_{0\leq i<j\leq n}(x^i\omega_j-\omega_ix^j)s_{i,j}
.\]
Hence,
\[
    |\mathfrak p_x(\omega)|
        =
    \max_{0\leq i<j\leq n}|x_i\omega_j-\omega_ix_j|\cdot|x|^{-1}\cdot|\omega|^{-1}
        =
    \|\omega-x\|
.\]
\end{proof}

\begin{Lem}\label{Lemma zero dimensional ln|p(omega)|}
Let $\mathfrak p$ be a $0$-dimensional homogeneous prime ideal over $K$. Then $\mathfrak p$ has $\deg\mathfrak p$ unique zeros in $\overline{{\mathbb Q}}\mathbb{P}^n$, which we denote by $\omega_1,...,\omega_{\deg\mathfrak p}$ (we choose representatives in $\overline{\mathbb Q}^{n+1}-\{0\}$), and it satisfies
\[
    \forall\omega\in\mathbb C^{n+1}-\{0\}:\quad
    \ln|\mathfrak p(\omega)|
        \geq
    \sum_i \ln\|\omega-\omega_i\|
        -
    C\cdot\deg\mathfrak p
,\]
for some constant $C$ depending only on $n$.
\end{Lem}

\begin{proof}
Let $\mathfrak p^{\overline{\mathbb Q}}$ be the ideal that $\mathfrak p$ generates over $\overline{\mathbb Q}$. Note that $\mathfrak p$ and $\mathfrak p^{\overline{\mathbb Q}}$ have the same Chow form (and thus the same degree, with $|\mathfrak p(\omega)|=|\mathfrak p^{\overline{\mathbb Q}}(\omega)|$).
By \iref{Remark decomposition of prime ideal over K}, $\mathfrak p^{\overline{\mathbb Q}}$ decomposes as
\[
    \mathfrak p^{\overline{\mathbb Q}}=\bigcap_{i=1}^{\deg\mathfrak p}\mathfrak p_i^{\overline{\mathbb Q}}
,\]
where the $\mathfrak p_i^{\overline{\mathbb Q}}$ are prime ideals over $\overline{\mathbb Q}$ of dimension $0$ and are conjugates of each other under $\operatorname{Gal}(\overline{\mathbb Q},  K)$. Each $\mathfrak p_i^{\overline{\mathbb Q}}$ also has degree $1$ and a single zero in $\overline{\mathbb Q}\mathbb P^n$, which we denote by $\omega_i$.

Applying \iref{Lemma Nesterenko primary decomposition of ideal} to the decomposition of $\mathfrak p^{\overline{\mathbb Q}}$ (which takes place over an appropriate number field $K[\omega_1,...,\omega_{\deg p}]$), we obtain, for any $\omega\in\mathbb C^{n+1}-\{0\}$,
\[
    \sum_i\ln|\mathfrak p_i^{\overline{\mathbb Q}}(\omega)|
        \leq
    \ln|\mathfrak p^{\overline{\mathbb Q}}(\omega)|+n^3\cdot\deg\mathfrak p^{\overline{\mathbb Q}}
.\]
By \iref{Lemma Chow form of single point},
\[
    |\mathfrak p_i^{\overline{\mathbb Q}}(\omega)|=\|\omega-\omega_i\|
,\]
and hence,
\begin{align*}
    \forall\omega\in\mathbb C^{n+1}-\{0\}:&
    \\&\ln|\mathfrak p(\omega)|+\mathrm{const}\cdot\deg\mathfrak p
        \geq
    \sum_i\ln|\mathfrak p_i^{\overline{\mathbb Q}}(\omega)|
        =
    \sum_i\ln\|\omega-\omega_i\|
.\end{align*}
\end{proof}

\begin{Lem}[Upper bound on $\rho$]\label{Lemma upper bound on rho}
Let $I\subseteq K[x_0,...,x_n]$ be a homogeneous unmixed ideal and let $\omega_0\in\mathbb C^{n+1}-\{0\}$ satisfy
\[
    \ln|I(\omega_0)|
        <
    -2([ K:\mathbb Q]+3)n^4\cdot\operatorname{t}(I)
.\]
Then
\[
        \exists x_{\omega_0} \in \operatorname{V}(I):\quad
       \|x_{\omega_0}-\omega_0\|
            <
        e^{\frac{\ln|I(\omega_0)|}{2n\deg I}}
.\]
\end{Lem}

\begin{proof}
This follows by direct computation from \iref{Lemma Nesterenko zero of ideal close to omega}:
\begingroup\small
\begin{align*}
    \exists x_{\omega_0} \in \operatorname{V}(I):&
    \\&\deg I\cdot \ln \|\omega_0 - x_{\omega_0}\| 
        \leq
    \frac{\ln|I(\omega_0)|}{n}
    + ([ K:\mathbb Q] + 3) n^3\cdot\operatorname{t}(I)
        <
    \frac{\ln|I(\omega_0)|}{2n}
.\end{align*}
\endgroup
\end{proof}

\subsection{Calculations for the Bernstein index of an ideal}\label{Subsection order of primary decomposition}

\begin{Lem}
\label{Lemma Bernstein index of I}
Let $I\subseteq K[x_0,...,x_n]$ be a non-trivial homogeneous unmixed ideal. Then
\[
    \B_{\frac1e,1}(|I|\circ\gamma)
        \leq
    \operatorname*{ord}_{\Gamma}I
    +
    C\cdot\deg I
,\]
where $C$ is a constant depending only on $n$.
\end{Lem}

\begin{proof}
Using \iref{Lemma bound on ln|I(omega)|} and the fact that $\operatorname*{ord}_{\Gamma}I:=-\ln\max_{\omega\in\Gamma}|I(\omega)|$, we obtain
\[
    \B_{\frac1e,1}(|I|\circ\gamma)
        :=
    \ln\max_{\omega\in\Gamma_1}|I(\omega)|
    -
    \ln\max_{\omega\in\Gamma}|I(\omega)|
        \leq
    \mathrm{const}\cdot\deg I
    +
    \operatorname*{ord}_{\Gamma}I
.\]
\end{proof}

\begin{Rem}\label{Remark Bernstein index of p}
The same proof applies to a polynomial $p\in K[x_0,...,x_n]$:
\[
    \B_{\frac1e,1}(p\circ\gamma)
        \leq
    \operatorname*{ord}_{\Gamma}p
    +
    C\cdot\operatorname{t}(p)
.\]
where $C$ is a constant depending only on $\Gamma$.
Use the simple bound $\ln|p(\omega)|\leq(\mathrm{const}+\max(0,\ln|\omega|))\cdot\operatorname{t}(p)$ in place of \iref{Lemma bound on ln|I(omega)|}.
\end{Rem}

\begin{Lem}\label{Lemma lower bound on I in general}
There exists a constant $C>0$ depending only on $n$ with the following property:
Let $I\subseteq K[x_0,...,x_n]$ be a non-trivial homogeneous unmixed ideal, and let $0<h<\frac1e$.
Then there are finitely many discs $\{D^i\}_{i=1}^a$ with sum of diameters less than $h$ and with $a\leq C\cdot(\max(0,\operatorname*{ord}_{\Gamma}I)+\deg I)$ such that
\[
    \forall x\in\overline{D_\frac1e}-\bigcup_i D^i:\quad
    \ln|I(\gamma(x))|
        \geq
-C\cdot\ln\frac1h\cdot(\max(0,\operatorname*{ord}_{\Gamma}I)+\deg I)
.\]
\end{Lem}

\begin{proof}
Consider $\omega\in\Gamma_1$. We have assumed that $\max_i\omega_i=\omega_0=1$. Thus, as a function on $\Gamma_1$, $|I(\omega)|$ is by definition the maximum of the absolute values of polynomials in the coordinates $\{w_i\}_{i\geq 1}$. Recall the definition $(|I|\circ\gamma)(z):=|I(\gamma(z))|$.

Thus, applying \iref{Lemma Bernstein index lower bound for max of holomorphic functions} with $f:=|I|\circ\gamma$, we obtain finitely many discs $\{D^i\}_{i=1}^a$ with sum of diameters less than $h$ and with $a\leq\mathrm{const}\cdot\B_{\frac1e,1}(|I|\circ\gamma)$ such that
\[
    \forall x\in\overline{D_\frac1e}-\bigcup_i D^i:
        \quad
    |I(\gamma(x))|
        \geq
    \frac{
        \max_{z\in\overline{D_\frac1e}}|I(\gamma(z))|
    }{
        e^{
            \B_{\frac1e,1}(|I|\circ\gamma)\cdot\mathrm{const}\cdot\ln\frac1h
        }
    }
.\]
Using \iref{Lemma Bernstein index of I} and $\operatorname*{ord}_{\Gamma}I:=-\ln\max_{z\in\overline{D_\frac1e}}|I(\gamma(z))|$, we obtain
\[
    a
        \leq
    \mathrm{const}\cdot\operatorname*{ord}_{\Gamma}I
    +
    \mathrm{const}\cdot\deg I
.\]
and
\begin{align*}
    \forall x \in \overline{D_{\frac1e}} - \bigcup_i D^i&: \quad\\
    \ln |I(\gamma(x))| &\geq \ln \max_{z \in \overline{D_{\frac1e}}} |I(\gamma(z))|
    - \B_{\frac1e, 1}(|I| \circ \gamma)\cdot\mathrm{const}\cdot\ln\frac1{h}\\
    &\geq - \operatorname*{ord}_{\Gamma} I
    - \left( \operatorname*{ord}_{\Gamma} I + \mathrm{const} \cdot \deg I \right) \cdot\mathrm{const}\cdot\ln\frac1{h} \\
    &= - \left(1+\mathrm{const}\cdot\ln\frac1{h}\right)\operatorname*{ord}_{\Gamma}I-\mathrm{const}\cdot\ln\frac1{h}\cdot\deg I
.\end{align*}
\end{proof}

\begin{Rem}\label{Remark lower bound on p in general}
The same proof applies to a polynomial $p\in K[x_0,...,x_n]$ and any $0<h<\frac1e$:
There exists a constant $C>0$ depending only on $\Gamma$ with the following property:
There are finitely many discs $\{D^i\}_{i=1}^a$ with sum of diameters less than $h$ and with $a\leq C\cdot(\max(0,\operatorname*{ord}_{\Gamma}p)+\operatorname{t}(p))$ such that
\[
    \forall x\in\overline{D_\frac1e}-\bigcup_i D^i:\quad
    \ln|p(\gamma(x))|
        \geq
-C\cdot\ln\frac1h\cdot(\max(0,\operatorname*{ord}_{\Gamma}p)+\operatorname{t}(p))
.\]
Use \iref{Remark Bernstein index of p} in place of \iref{Lemma Bernstein index of I}.
\end{Rem}

\subsection{Calculations for functions and derivatives}

\begin{Lem}[Effect of differentiation on $|\cdot|_v$, $\deg$, and $\operatorname{H}$]\label{Lemma |xi(f)|_v, deg, H}
Let $f\in K[x_0,...,x_n]$ be a homogeneous polynomial such that $\xi(f)\not\equiv 0$, and let $|\cdot|_v\in\mathcal V$. Let $A\in\mathbb N$ be such that $\hat{\xi}:=A\cdot\xi$ has only algebraic integer coefficients. Then there exists $C$ depending only on $\xi,K,A$ such that:
\begin{enumerate}
\item For non-Archimedean $v$: $|\hat{\xi}(f)|_v\leq |f|_v$.
\item For Archimedean $v$: $|\hat{\xi}(f)|_v\leq C\cdot\deg f\cdot|f|_v$.
\item $\deg\xi(f)\leq\deg f+C$.
\item $\operatorname{H}(\xi(f))\leq C\cdot(\deg f)^C\cdot\operatorname{H}(f)$.
\end{enumerate}
\end{Lem}

\begin{Rem}
If we choose $A$ to be the smallest such $A$, then $C$ depends only on $\xi,K$.
\end{Rem}

\begin{proof}
\hfill
\begin{enumerate}
\item Since the coefficients of $\hat{\xi}$ are algebraic integers, each step in calculating $\hat{\xi}(f)$ can only decrease the norm (when it is non-Archimedean).
\item Calculating $\hat{\xi}(f)$ involves adding a constant number of terms of the form $a_{i_0,...,i_n,j}x_0^{i_0}\cdots x_n^{i_n}\frac{\partial f}{\partial x_j}$. Since there are only a constant number of Archimedean norms, each such term has norm at most $\mathrm{const}\cdot\deg f\cdot|f|_v$.
\item This is clear.
\item Since $\operatorname{H}(\xi(f))=\operatorname{H}(\hat{\xi}(f))$, the previous bounds give $\operatorname{H}(\hat{\xi}(f)):=\prod_v|\hat{\xi}(f)|_v\leq(\mathrm{const}\cdot\deg f)^\mathrm{const}\cdot\operatorname{H}(f)$.
\end{enumerate}
\end{proof}

\begin{Cor}[Effect of high-order differentiation on $\operatorname{t}$]\label{Corollary t(xi^a(f))<t(f)+a*ln(t(f)+a)}
Let $f\in K[x_0,...,x_n]$ be a homogeneous polynomial and let $a\in\mathbb N^{\geq 1}$ be such that $\xi^a(f)\not\equiv 0$. Then
\[
    \operatorname{t}(\xi^a(f))
        \leq
    \operatorname{t}(f)+C\cdot a\cdot\ln(\deg f+C\cdot a)
.\]
where $C$ is a constant depending only on $\xi,K$.
\end{Cor}

\begin{proof}
This follows by direct calculation.
\end{proof}

\begin{Lem}\label{Lemma |xi(f)|>e^-(const*t(f))|f|}
Let $f\in K[x_0,...,x_n]$ be a homogeneous polynomial such that $\xi(f)\not\equiv 0$. Then
\[
    |\xi(f)|
        \geq
    e^{-C\cdot\operatorname{t}(f)}|f|
,\]
where $C$ is a constant depending only on $\xi,K$.
\end{Lem}

\begin{proof}
Let
\begin{itemize}
    \item $g:=\xi(f)$.
    \item $|\cdot|_\mathbb C=:|\cdot|$ be the usual Euclidean absolute value on $\mathbb C$.
    \item $f_i$ be the nonzero coefficients of $f$.
    \item $g_i$ be the nonzero coefficients of $g$.
\end{itemize}

Recall that
\[
    \operatorname{H}(f):=\prod_v|f|_v=|f|_\mathbb C\prod_{v\neq|\cdot|_\mathbb C}|f|_v
.\]
Using \iref{Lemma |xi(f)|_v, deg, H}, we obtain
\begin{align*}
    \forall i : \quad
    1 &= \prod_v |A\cdot g_i|_v 
      = |A\cdot g_i|_\mathbb{C} \cdot \prod_{v \neq |\cdot|_\mathbb{C}} |A\cdot g_i|_v \\
      &\leq |A\cdot g_i|_\mathbb{C} \cdot \prod_{v \neq |\cdot|_\mathbb{C}} |A\cdot g|_v
\leq |A\cdot g_i|_\mathbb{C} \cdot \mathrm{const} \cdot (\deg f)^{\mathrm{const}} \prod_{v \neq |\cdot|_\mathbb{C}} |f|_v \\
      &=|A\cdot g_i|_\mathbb{C} \cdot \mathrm{const} \cdot (\deg f)^{\mathrm{const}} \cdot \frac{\operatorname{H}(f)}{|f|_\mathbb{C}}
      \leq|g_i|_\mathbb{C} \cdot \frac{e^{\mathrm{const} \cdot \operatorname{t}(f)}}{|f|_\mathbb{C}}
.\end{align*}
Hence,
\[
    |g|
        =
    \max_i|g_i|_\mathbb C
        \geq
    e^{-\mathrm{const}\cdot\operatorname{t}(f)}|f|
.\]
\end{proof}

\begin{Lem}\label{Lemma f(omega)<|f|*rho}
Let $f\in\mathbb C[x_0,...,x_n] -\{0\}$ be homogeneous and let $x,\omega\in\mathbb C^{n+1}-\{0\}$ be such that $f(x)=0$. Then
\[
    \|f\|_\omega
        \leq
    e^{C\cdot\deg f}\cdot\|\omega-x\|
,\]
where $C$ is a constant depending only on $n$.
\end{Lem}

\begin{proof}
Let $i_\mathrm{max}$ be the maximal index of $x$. Define $a:=x_{i_\mathrm{max}}\cdot\omega$ and $b:=\omega_{i_\mathrm{max}}\cdot x$. We use the substitution $f(a)=f(b+(a-b))$.

We write each term of $f(a)$ as a coefficient times a product $[b_{i_1}+(a_{i_1}-b_{i_1})]\cdots[b_{i_{\deg f}}+(a_{i_{\deg f}}-b_{i_{\deg f}})]$ and expand the brackets to obtain $2^{\deg f}$ summands from each term of $f$. The summands that contain any $(a_i-b_i)$ are bounded by $|f|\cdot\|a-b\|_\infty\cdot(\|a\|_\infty+\|b\|_\infty)^{\deg f-1}$, and the summands that contain only $b_i$'s give exactly $f(b)=0$. Overall, we obtain
\begin{align*}
    |x|^{\deg f}|f(\omega)|
        &=
    |f(a)|
        =
    |f(a+(b-a))|\\
        &\leq
   (n+1)^{\deg f}\cdot 2^{\deg f}\cdot|f|\cdot\|a-b\|_\infty\cdot(\|a\|_\infty+\|b\|_\infty)^{\deg f-1}
.\end{align*}

Note that
\[
    \|a-b\|_\infty
        =
    \max_j\{|x_{i_\mathrm{max}}\cdot\omega_j-\omega_{i_\mathrm{max}}\cdot x_j|\}
        \leq\
    \|\omega-x\|\cdot |\omega|\cdot|x|
,\]
and that
\[
    \|a\|_\infty+\|b\|_\infty
        =
    |x_{i_\mathrm{max}}|\cdot|\omega|+|\omega_{i_\mathrm{max}}|\cdot|x|
        =
    2\cdot|x|\cdot|\omega|
,\]
and hence,
\[
    \|f\|_\omega
        \leq
    e^{\mathrm{const}\cdot\deg f}\cdot\|\omega-x\|
.\]
\end{proof}

\begin{Lem}\label{Lemma bound on xi(f)(omega)}
Let $f\in K[x_0,...,x_n]$ be a homogeneous polynomial such that $\xi(f)\not\equiv 0$. Then there exists $A\subseteq\overline{D_\frac1e}$ with $\operatorname{Area}(A)\geq 0.56\cdot\operatorname{Area}\left(\overline{D_\frac1e}\right)$ such that
\begin{align*}
    \forall\omega\in\gamma(A)&:\\
    \|\xi&(f)\|_\omega
        \leq
    e^{C\cdot\operatorname{t}(f)}\cdot\left(\|f\|_\omega+\|f\|_\omega^{0.6}\right)
,\end{align*}
where $C$ is a constant depending only on $\xi,\gamma,K$.
\end{Lem}

\begin{proof}
By definition, $\xi_{\gamma(z)}=a_z\frac{d\gamma}{dz}(z)$, where $a_z$ depends holomorphically on $z$, so
\[
    \forall z \in \overline{D_1} : \quad
    \xi(f)(\gamma(z))
        =
    \xi_{\gamma(z)}(f) 
        =
    a_z\cdot(f \circ \gamma)'(z)
.\]
Since $a_z$ is holomorphic in $z$, $|a_z|$ attains a maximum on $\overline{D_\frac1e}$. Thus,
\[
    \forall z\in\overline{D_1}:\quad
    |\xi(f)(\gamma(z))|
        \leq
    \mathrm{const}\cdot|(f\circ\gamma)'(z)|
.\]

Furthermore, by \iref{Proposition derivative-function ratio}, there exists $A\subseteq\overline{D_\frac1e}$ with $\operatorname{Area}(A)\geq 0.56\cdot\operatorname{Area}\left(\overline{D_\frac1e}\right)$ such that
\[
    \forall z\in A:\quad
    |(f\circ\gamma)'(z)|
        \leq
    \mathrm{const}\cdot\B_{\frac1e,1}(f\circ\gamma)|(f\circ\gamma)(z)|
.\]
We therefore obtain
\[
    \forall\omega\in\gamma(A):\quad
    |\xi(f)(\omega)|
        \leq
    \mathrm{const}\cdot\B_{\frac1e,1}(f\circ\gamma)|f(\omega)|
.\]

Using the definition of the Bernstein index and the bounds $
    \ln\max_{x\in\Gamma_1}|f(x)|
        \leq
    \mathrm{const}\cdot\deg f+\ln|f|
$ and $
    \forall\omega\in\gamma(A):\;
    \ln\max_{x\in\Gamma}|f(x)|\geq\ln|f(\omega)|
$
we obtain
\begin{align*}
    \forall\omega\in\gamma(A)&:\\
    &\frac{|\xi(f)(\omega)|}{|f|}
        \leq
    \begin{cases}
    \mathrm{const}\cdot\left(\deg f+\max\left(0,\ln\frac{|f|}{|f(\omega)|}\right)\right)\frac{|f(\omega)|}{|f|} & \mathrm{if }f(\omega)\neq 0\\
    0                                  &\mathrm{if }f(\omega)=0
\end{cases}
,\end{align*}
Finally, since $
    \forall x>0:\,
    x\ln\frac1x
        \leq
    x^{0.6}
$ and $|\omega|=1$, \iref{Lemma |xi(f)|>e^-(const*t(f))|f|} gives
\begin{align*}
    \forall\omega\in\gamma(A)&:\\
    \|\xi&(f)\|_\omega
        \leq
    e^{\mathrm{const}\cdot\operatorname{t}(f)}\cdot\left(\|f\|_\omega+\|f\|_\omega^{0.6}\right)
.\end{align*}
\end{proof}

\bibliographystyle{plain}
\bibliography{references}
\end{document}